\documentclass[11pt,reqno]{amsart}
\usepackage{xcolor}
\usepackage{mathrsfs}
\usepackage{amsmath,amsthm,amssymb,amscd}
\usepackage{booktabs} 
\usepackage{float}
\usepackage[hyperfootnotes=false, colorlinks, linkcolor={blue}, citecolor={magenta}, filecolor={blue}, urlcolor={blue}, plainpages=false, pdfpagelabels]{hyperref}
\usepackage[overload]{textcase} 
\usepackage{mathtools}
\usepackage[numbers,sort&compress]{natbib} 
\usepackage{etoolbox} 
\apptocmd{\sloppy}{\hbadness 10000\relax}{}{} 
\usepackage[left=2.5cm, right=2.5cm, top=3cm]{geometry}
\allowdisplaybreaks
\usepackage{enumerate}
\usepackage[shortlabels]{enumitem}
\usepackage{bm}
\usepackage{url}
\usepackage{soul}
\usepackage{cancel}
\usepackage{constants}
\newconstantfamily{abcon}{symbol=c}

\usepackage[multiple]{footmisc}

\usepackage{colonequals}
\usepackage{longtable} 
\usepackage[ampersand]{easylist}
\usepackage{ltablex,booktabs}
\usepackage{cleveref}

\usepackage{tikz}
\usepackage{tikz-cd}

\usepackage{bbm}
\newtheorem{theorem}{Theorem}[section]
\newtheorem{corollary}[theorem]{Corollary}

\newtheorem{proposition}[theorem]{Proposition}
\newtheorem{lemma}[theorem]{Lemma}

\theoremstyle{definition}

\theoremstyle{remark}
\newtheorem{remark}{Remark}

\numberwithin{equation}{section}

\crefname{figure}{Figure}{Figures}
\theoremstyle{plain}
\newtheorem*{theorem*}{Theorem}
\crefname{theorems}{Theorem}{Theorems}
\crefname{corollaries}{Corollary}{Corollaries}
\newtheorem*{corollary*}{Corollary}
\crefname{corollaries*}{Corollary}{Corollaries}
\crefname{lemma}{Lemma}{Lemmata}
\crefname{proposition}{Proposition}{Propositions}
\crefname{conjectures}{Conjecture}{Conjectures}
\newtheorem*{conjonjecture*}{Conjecture}
\crefname{conjonjectures*}{Conjecture}{Conjectures}
\crefname{definitions}{Definition}{Definitions}
\crefname{hypotheses}{Hypothesis}{Hypotheses}

\def\im{\operatorname{Im}}
\def\re{\operatorname{Re}}
\def\add{\text{add}}

\def\Aut{\operatorname{Aut}}

\def\det{\operatorname{det}}

\newcommand{\A}{{\mathbb A}}
\newcommand{\Q}{{\mathbb Q}}
\newcommand{\Z}{{\mathbb Z}}

\newcommand{\CC}{{\mathbb C}}

\newcommand{\F}{{\mathbb F}}

\newcommand{\p}{\mathfrak p}

\newcommand{\GL}{{\rm GL}}

\newcommand{\SL}{{\rm SL}}

\newcommand{\rad}{{\rm rad}}

\newcommand{\tr}{{\rm tr}}
\newcommand{\Frob}{{\rm Frob}}
\def\p#1{\left( #1 \right)}
\def\kronecker#1#2{\p{\frac{#1}{#2}}}

\newcommand{\Gal}{{\rm Gal}}

\newcommand{\Hom}{{\rm Hom}}

\newcommand{\diag}{{\rm diag}}

\newcommand{\RN}[1]{%
  \textup{\uppercase\expandafter{\romannumeral#1}}%
}

\newcommand{\ol}[1]{\overline{#1}}

\newcommand{\forget}[1]{}
\def\qdots{\mathinner{\mkern1mu\raise0pt\vbox{\kern7pt\hbox{.}}\mkern2mu
\raise3.4pt\hbox{.}\mkern2mu\raise7pt\hbox{.}\mkern1mu}}

\usepackage[toc,page, title, titletoc]{appendix}
\makeatother
\makeatletter
\newcommand\appendix@section[1]{%
	\refstepcounter{section}%
	\orig@section*{Appendix \@Alph\c@section: #1}%
	\addcontentsline{toc}{section}{Appendix \@Alph\c@section: #1}%
}
\g@addto@macro\appendix{\let\section\appendix@section}
\let\orig@section\section
\makeatother
\makeatletter
\def\thickhline{%
  \noalign{\ifnum0=`}\fi\hrule \@height \thickarrayrulewidth \futurelet
   \reserved@a\@xthickhline}
\def\@xthickhline{\ifx\reserved@a\thickhline
               \vskip\doublerulesep
               \vskip-\thickarrayrulewidth
             \fi
      \ifnum0=`{\fi}}
\makeatother

\title{Effective open image theorem and a Linnik type problem for elliptic curves}
\author[Wang]{Tian Wang}
\address{Department of Mathematics \&
Statistic, Concordia University, Montreal, Canada}
\email{tian.wang@concordia.ca}
\author[Wei]{Zhining Wei}
\address{Department of Mathematics, Brown University, Providence, USA}
\email{zhining\_wei@brown.edu}
\date{}

\begin{document}
\maketitle

\begin{abstract}
 We study an effective open image theorem for families of elliptic curves and products of elliptic curves  ordered by conductor. Unconditionally, we prove that for $100\%$ of pairs of elliptic curves, the largest prime $\ell$, for which the associated mod $\ell$ Galois representation fails to be surjective, is small. Additionally, for semistable families, our bound on $\ell$ is comparable to the result obtained under the Generalized Riemann Hypothesis. We reduce the problem to a Linnik type problem for modular forms and  apply the zero density estimates. This method, together with an analysis of the local Galois representations of elliptic curves, allows us to show similar results for single elliptic curves.

\end{abstract}

\section{Introduction}

Let $E/\Q$ be a non-CM elliptic curve with conductor $N_E$. For a prime number $\ell$, let $E[\ell]$ denote the $\ell$-torsion subgroup of $E(\overline{\Q})$ and $T_\ell(E)$ denote the $\ell$-adic Tate module of $E$. Upon fixing bases, we obtain the following isomorphisms of modules:  \footnote{Below, $\F_{\ell}$ denotes the finite field with $\ell$ elements.}
\[
E[\ell] \simeq_{\F_\ell} \F_\ell \oplus \F_\ell \quad \text{and} \quad T_\ell(E) \simeq_{\Z_\ell} \Z_\ell \oplus \Z_\ell.
\]
The absolute Galois group $\Gal(\overline{\Q}/\Q)$ acts coordinate-wise on elements of $E[\ell]$ and $T_\ell(E)$,  giving rise to the mod $\ell$ Galois representation and the  $\ell$-adic Galois representation of $E$, denoted respectively by
\[
\bar{\rho}_{E, \ell}\colon\Gal(\overline{\Q}/\Q) \longrightarrow \GL_2(\F_{\ell}) \quad \text{and} \quad
\rho_{E, \ell}\colon\Gal(\overline{\Q}/\Q) \longrightarrow \GL_2(\Z_{\ell}).
\]
For a prime $p\nmid N_E$, we denote by  $\overline{E}_p(\F_p)$ the $\F_p$-points of the reduction of $E$ at $p$. The constant $a_p(E):=p+1-|\overline{E}_p(\F_p)|$ is referred to as the Frobenius trace of $E$ at $p$.
It is known that  
$\tr \rho_{E, \ell}(\Frob_p) = a_p(E)$ and $\det \rho_{E, \ell}(\Frob_p) = p$, where $\Frob_p \in \Gal(\overline{\Q}/\Q)$ is a Frobenius automorphism at $p$.

Serre's celebrated open image theorem \cite{Se1972} states that for any sufficiently large prime $\ell$, $\rho_{E,\ell}$ is surjective. For $\ell\geq 5$, the surjectivity of $\rho_{E,\ell}$ is equivalent to the surjectivity of $\bar{\rho}_{E,\ell}$. For this reason, we primarily work with $\bar{\rho}_{E,\ell}$ since it has finite image. We call $\ell$ a \textit{non-surjective prime} of $E$ if $\bar{\rho}_{E,\ell}$ is non-surjective. 

In this paper, we adopt the convention that $p$ and $\ell$ are primes and we assume $p\neq \ell$. 
Define 
\[
c(E):=\min\{\textup{$p$: if  $\ell>p$, then  $\ol{\rho}_{E,\ell}$ is surjective}\}.
\]
The famous Serre's  uniformity question asks whether $c(E)\leq 37$ for all non-CM elliptic curves.
Significant progress has been made toward answering this question. It is known that for  $\ell>37$, if $\ol{\rho}_{E, \ell}$ is non-surjective,  then the image is equal (up to conjugation) to the normalizer of a non-split Cartan subgroup of $\GL_2(\F_\ell)$ (see \cite{Se1972, Ma1978, Se1981, BiPaRe2013, MR3961086, FuLo2023}). Recently, Lemos \cite{Lemos2019, LemosII2019} showed that if $E$ has a non-trivial cyclic isogeny defined over $\Q$ or if   there is a prime $\ell$ such that the image of $\ol{\rho}_{E, \ell}$ is contained in  the normalizer of a  split Cartan subgroup of $\GL_2(\F_\ell)$, then $c(E)\leq 37$. 
Additionally, Mazur \cite{Ma1978} proved that if $E$ is  semistable (i.e., $N_{E}$ is squarefree), then $c(E)\leq 7$. 
However, for general non-CM elliptic curves $E$, much less is known about uniform bounds of $c(E)$.

On the other hand, there are several unconditional bounds of $c(E)$ in terms of the invariants of $E$.  The work of Masser and W\"{u}stholz \cite{MaWu1993} and Lombardo \cite[Corollary 9.3]{Lo2015} showed that there exists absolute and effectively computable constants $\Cl[abcon]{MasserWustholz1}>0$ and $\Cl[abcon]{MasserWustholz2}>0$ such that  $c(E) \leq \Cr{MasserWustholz1}\max\{1, h_{E}^{\Cr{MasserWustholz2}}\}$, where $h_E$ is the stable Faltings height of $E$. Kraus \cite{Kr1995} and Cojocaru \cite{Co2005} 
 proved that there is an explicit constant $\Cl[abcon]{KrCo}>0$ such that   
\[c(E) \leq \Cr{KrCo} \rad(N_E) (1 + \log \log \rad(N_E))^{1/2},\] where $\rad(n):=\prod_{p\mid n} p$.  

Additionally, much better upper bounds for $c(E)$ are known under the assumption of the Generalized Riemann Hypothesis (GRH) for Dedekind zeta functions. Under this assumption, we have  $c(E)\ll \log \rad(N_E)$, where the implied constant is explicit  (see, e.g., \cite{ChSw2024}).

Next, we consider the Galois representations for a pair of elliptic curves.   Let $E_1/\Q$ and $E_2/\Q$ be two non-CM elliptic curves that are not $\overline{\Q}$-isogenous. Similar as before, the action of the absolute Galois group $\Gal(\overline{\Q}/\Q)$ on the $\ell$-torsion group $(E_1\times E_2)[\ell]\simeq E_1[\ell]\times E_2[\ell]$ gives rise to the  
mod $\ell$ Galois representation
\[
\bar{\rho}_{E_1\times  E_2,\ell}:  \Gal(\overline{\Q}/\Q) \to  \GL_2 (\F_\ell) \times \GL_2(\F_\ell).
\]
Because of the Weil pairings, the image of  $\ol{\rho}_{E_1\times E_2, \ell}$ lies in the proper subgroup
\[
(\GL_2\times_{\det}\GL_2)(\F_\ell) := \{(M_1, M_2): M_1, M_2\in \GL_2(\F_\ell),  \det(M_1)  = \det(M_2)\}
\]
The open image theorem in this case (see  \cite[Th\'eor\`eme 6, p. 324]{Se1972} or \cite[Theorem 3.5]{Ri1975}) implies that the quantity
\[
c(E_1\times E_2)=\min\{\textup{$p$: if $\ell>p$, then $\bar{\rho}_{E_1\times E_2,\ell}$ is surjective (on $(\GL_2\times_{\det}\GL_2)(\F_\ell))$}\}
\] 
exists and is finite. 

Compared to the case of single elliptic curves, fewer results are known about  $c(E_1\times E_2)$. Unconditionally, Masser and W\"{u}stholz \cite[Proposition 1, p. 251] {MaWu1993} and Lombardo \cite[Theorem 1.1]{Lombardo2016}  proved that there exist absolute and effectively computable  constants $\Cl[abcon]{MWL1}>0$ and $\Cl[abcon]{MWL2}>0$ such that $c(E_1\times E_2)\leq \Cr{MWL1}\max\{1,h_{E_1}^{\Cr{MWL2}},h_{E_2}^{\Cr{MWL2}}\}$.  Murty \cite[Proof of Theorem 8]{Mu1999} showed that there exists an absolute and effective constant $\Cl[abcon]{Murty1}>0$ such that $c(E_1\times E_2)\leq (\max\{N_{E_1},N_{E_2}\})^{\Cr{Murty1}}$ without giving the constant explicitly.  Assuming GRH, Mayle and Wang  recently \cite{MaWa2023b} showed that $c(E_1\times E_2) \ll  
\log \rad (N_{E_1}N_{E_2})$, where the constant in $\ll$ is explicit. 

The main objective of this paper is to improve the unconditional upper bounds of $c(E)$ and $c(E_1\times E_2)$, aiming to approach the bounds known under GRH, particularly when working with families of elliptic curves.

\subsection{Main results}
We begin by presenting the result for products of elliptic curves. 

 Observing that $\Q$-isogenous elliptic curves have the same conductor, we define $\mathcal{E}(N)$ to be the set of $\Q$-isogeny classes  
of non-CM 
 elliptic curves  $E/\Q$ with conductor at most $N$.   
 Similarly, we introduce $\mathcal{E}^{ss}(N)\subseteq \mathcal{E}(N)$, the set  of $\Q$-isogenous classes of semistable elliptic curves  with conductor at most $N$. We have the following  lower and upper  bounds of  $|\mathcal{E}^{ss}(N)|$ and $|\mathcal{E}(N)|$:  \footnote{The bounds in  \eqref{eq:family-size} are, a priori, for the number of both  non-CM and CM elliptic curves.   We will show in the Appendix that 
 the set $\mathcal{E}^{CM}(N)$ of $\Q$-isogeny classes of CM elliptic curves over $\Q$, with conductor bounded by $N$, satisfies the bound $|\mathcal{E}^{CM}(N)|\ll N^{\frac{1}{2}}$. Hence, the set $\mathcal{E}^{CM}(N)$  is negligible in our problem.
 }
\begin{equation}\label{eq:family-size}
N^{\frac{5}{6}} \ll |\mathcal{E}^{ss}(N)| \leq |\mathcal{E}(N)| \ll N^{1+\varepsilon} \quad  \forall \varepsilon>0,
\end{equation}
which come from \cite[Eq. (1.12)]{FoNaTe1992} and \cite[Proposition 1]{DuKo2000}, respectively.

Our first main theorem is as follows. 
 
\begin{theorem}
\label{unconditional for two elliptic curves thm family}
Let $\mathcal{E}'(N)\subseteq \mathcal{E}(N)$ be any subfamily such that $|\mathcal{E}'(N)|\gg N^{\delta}$ with $\frac{1}{2}<\delta\leq 1$. Then, for any $\varepsilon>0$, 
\[
\lim_{N\to \infty}\frac{\#\left\{(E_1, E_2)\in \mathcal{E}'(N) \times \mathcal{E}'(N)\colon c(E_1\times E_2)\leq \max_{i\in \{1, 2\}}\left\{c(E_i),  4(\log N_{E_i})^{c(\delta)/2+\varepsilon}\right\}\right\}}{|\mathcal{E}'(N)|^2}=1,
\]
where $c(\delta)$ is defined in \eqref{eq. constant c}.\footnote{The least value for $c(\delta)$, when $\frac{1}{2}<\delta\leq \frac{5}{6}$, is $c(5/6)=5558$.}
In particular, if $\mathcal{E}'(N)\subseteq \mathcal{E}^{ss}(N)$, then 
\[
\lim_{N\to \infty}\frac{\#\left\{(E_1, E_2)\in \mathcal{E}'(N) \times \mathcal{E}'(N)\colon c(E_1\times E_2)\leq \max_{i\in \{1, 2\}}\left\{4(\log N_{E_i})^{c(\delta)/2+\varepsilon}\right\}\right\}}{|\mathcal{E}'(N)|^2}=1.
\]
\end{theorem}
\begin{remark}
In the proof of Theorem \ref{unconditional for two elliptic curves thm family}, we obtain a power saving error term, that is,
\begin{align*}
&\#\left\{(E_1, E_2)\in \mathcal{E}'(N) \times \mathcal{E}'(N)\colon c(E_1\times E_2)  \leq \max_{i\in \{1, 2\}}\left\{c(E_i),  4(\log N_{E_i})^{c(\delta)/2+\varepsilon}\right\}\right\}  \\
& \hspace{110mm} =|\mathcal{E}'(N)|^2+O(|\mathcal{E}'(N)|^{2-\varepsilon}). 
\end{align*}
\end{remark}

 Roughly speaking, Theorem \ref{unconditional for two elliptic curves thm family} is saying that, for $100\%$ pairs of elliptic curves $(E_1,E_2)$ ordered by conductor, $c(E_1\times E_2)$ has a log-type bound in terms of their conductors. Note that this bound is comparable to the one obtained assuming GRH.

For a product  of $n\geq 2$ elliptic curves $E_1, \ldots, E_n$ over $\Q$ that are non-CM and  pairwise not $\overline{\Q}$-isogenous, we have (see \cite[Section 5, (5.1)]{MaWa2023b})  
\[
c(E_1\times \cdots \times E_n) \leq \max_{1\leq i<j\leq n}\{c(E_i\times E_j)\}.
\]
 Therefore, \Cref{unconditional for two elliptic curves thm family}  can be readily extended to a family of  $n$-tuples  elliptic curves.

\begin{remark}\label{rm:watkins} It is important to note that there are alternative orderings of elliptic curves, such as by naive height. When the pairs $(E_1, E_2)$ are ordered by naive height, Jones  \cite{Jones2013} proved that $c(E_1\times E_2)=1$ for almost all pairs of elliptic curves. 
    In contrast, the arithmetic statistics of elliptic curves ordered by conductor exhibit more subtle behavior. Notably, it remains a conjecture of Watkins \cite[Heuristic 4.1]{Wa2008} that $|\mathcal{E}(N)|\sim \Cl[abcon]{Watkins} N^{\frac{5}{6}}$ for some explicit constant $\Cr{Watkins}>0$.  Despite the challenges, the study of how elliptic curve invariants are distributed when ordered by conductor has drawn considerable attention, as seen in works such as \cite{ShSHWa2021, Xi2024}.  
\end{remark}

\begin{remark}
  We expect that  \Cref{unconditional for two elliptic curves thm family} can be extended to an effective open image theorem for families of pairs of non-CM newforms. Denote by $\mathscr{F}(N)$ a set of weight $k$  holomorphic newforms  with conductor at most $N$ and trivial Nebentypus. For $f_1, f_2\in \mathscr{F}(N)$, let  $N_1$ and  $N_2$ be their conductors, respectively, and (for simplicity) assume they have a common coefficient field.    We denote by $\ol{\rho}_{f_i, \ell}$ (resp.  $\ol{\rho}_{f_1\times f_2, \ell}$)  the mod $\ell$ Galois representation of $f$ (resp.  the Rankin-Selberg convolution $f_1 \times f_2$). Analogous open image theorems exist for these representations (see, e.g.,  \cite[Theorem 2.3.1 and Theorem 3.4.1]{Loeffler2017}).  
Similar to the case of elliptic curves, we define the constants:
 \[
 c(f_i):=\min\{\textup{$p$: if  $\ell>p$, then  $\ol{\rho}_{f_i,\ell}$ is surjective}\} \quad i\in \{1, 2\},
 \]
 \[ 
 c(f_1\times f_2):=\min\{\textup{$p$: if  $\ell>p$, then  $\ol{\rho}_{f_1\times f_2,\ell}$ is surjective}\}.
 \]

The problem of determining upper bounds of $c(f_1\times f_2)$ can be reduced to a Linnik type problem.
 Since the Linnik-type result (i.e., \Cref{thm. multi}) applies in this setting, we deduce that for
 $100\%$ pairs of newforms  $(f_1,f_2)\in \mathscr{F}(N)\times \mathscr{F}(N)$ with $|\mathscr{F}(N)|\geq N^{\delta}$ for some  $\frac{1}{2}<\delta\leq 1$, 
 \[
 c(f_1\times f_2) \ll_{\epsilon} \max_{i\in \{1, 2\}}\{c(f_i),  (\log N_{f_i})^{c(\delta, k, \varepsilon)}\},
 \]
 where $c(\delta, k, \varepsilon)$ is an explicit function of  $\delta, k$, and $ \varepsilon$. 
\end{remark}

The statement of \Cref{unconditional for two elliptic curves thm family} naturally leads to the following question: can we obtain an unconditional upper bound of the form  
\begin{equation}\label{eq:cE-expectation}
  c(E)\ll (\log N_E)^\gamma, \text{ $\gamma>0$ is an absolute constant }  
\end{equation}
for almost all elliptic curves $E\in \mathcal{E}(N)$? This bound is comparable to what is known under GRH.

However, using our current techniques, proving \eqref{eq:cE-expectation} is more challenging than proving the corresponding result for pairs of elliptic curves. One reason is that our approach will inevitably confront the hard problem of understanding the local representations at primes where $E$ has additive, potentially good reductions. Nonetheless, we derive some unconditional results that push the boundaries of what is currently known about $c(E)$.

The first result is an  improvement of Kraus and Cojocaru's bound. 

\begin{theorem}\label{thm-uncod}
Let $E/\Q$ be a non-CM elliptic curve of conductor $N_E$. Then for any $\varepsilon>0$, we have $c(E)\ll_{\varepsilon}  N_E^{\frac{1}{2}+\varepsilon}$.
\end{theorem}
One observe that this result is still far from the bound \eqref{eq:cE-expectation}. Therefore, we hope to prove  \eqref{eq:cE-expectation} when $E$ lies in a subfamily of $\mathcal{E}(N)$.
To properly state this result, we begin by introducing some notation. Consider the following assumption on $E$:
\begin{equation}\label{inert-cond}
    \text{ For each prime $p\not\in \{2, 3\}$ of additive, potentially good reduction of $E$, we have  $|\Phi_p|\neq 4$,}
\end{equation}
where  $\Phi_p$ is a finite group that measures the failure of semistability of $E$ at $p$. This is  defined in \S \ref{subsec:inertial}. 
\footnote{We will see in \S \ref{subsec:inertial} that    $|\Phi_p|\neq 4$ is  equivalent to the following statements: the $p$-adic valuation $v_p(\Delta)$ of the minimal discriminant $\Delta$ of $E$ is 3 or 9;  the Kodaira Symbol of $E$ at $p$ is not III or III$^*$;  the  inertial Weil-Deligne type  $\tau_E$ of $E$ at $p$ is not isomorphic to $\tau_{ps}(1,1,4)$ or $\tau_{sc}(u,2,4)$.} Now, we define the subfamily 
\[
\mathcal{E}^{add}(N):=\{E\in \mathcal{E}(N)\colon \text{$E$ satisfies \eqref{inert-cond}}\}.
\]
Observe that $\mathcal{E}^{ss}(N)\subseteq \mathcal{E}^{add}(N)\subseteq \mathcal{E}(N)$. Finally, we  define a constant closely related to $c(E)$:
\[
c_0(E) := \min\{\text{$p$ prime: if $\ell$  satisfies \eqref{c0E} and $\ell>p$, then $\ol{\rho}_{E,\ell}$ is surjective}\}, \footnote{The condition of \eqref{c0E} on $\ell$ can be removed if one can show  $\ol{\rho}_{E, \ell}$ is surjective for all $\ell \mid N_E^{\add}$.}
\]
where the condition \eqref{c0E} for $\ell$ is: 
\begin{equation}
    \label{c0E}
\text{$\ell$ is not a prime of additive, potentially good reduction of $E$}.
\end{equation}

The following result shows that, for $100\%$ elliptic curves $E$ ordered by conductor, $c_0(E)$ has a log-type bound in terms of $N_E$.

\begin{theorem}
\label{thm-uncond-family inertia}
Let $\mathcal{E}'(N)\subseteq \mathcal{E}^{add}(N)$ be a subfamily of elliptic curves such that $|\mathcal{E}'(N)|\gg  N^{\delta}$ with $\frac{1}{2}< \delta \leq 1$.  
Then, for any $\varepsilon>0$, 
\[
\lim_{N\to \infty}\frac{\#\{E\in \mathcal{E}'(N): c_0(E)\leq (\log N_E)^{c(\delta)/2+\varepsilon}\}}{ |\mathcal{E}'(N)|}=1,
\]
where $c(\delta)$ is the constant defined in \eqref{eq. constant c}. Moreover, we can show:
\[\#\{E\in \mathcal{E}'(N): c_0(E)\leq (\log N_E)^{c(\delta)/2+\varepsilon}\}=|\mathcal{E}'(N)|+O(|\mathcal{E}'(N)|^{1-\varepsilon}).\]

\end{theorem}

\begin{remark}
 Since we do not know how to effectively count elliptic curves ordered by conductor, it is difficult to compare the family   $\mathcal{E}^{add}(N)$ with  $\mathcal{E}^{ss}(N)$. In this remark, we provide a brief argument showing that  $\mathcal{E}^{add}(N)$ contains a subfamily $\mathcal{E}_{\text{III}}^{sf}(N)\neq \mathcal{E}^{ss}(N)$ and is not too small in  $\mathcal{E}(N)$, assuming  Watkins's Conjecture. 

  Before defining the subfamily,   we recall some notation from \cite{ShSHWa2021}. Let $ \mathcal{E}_{\text{III}}^{sf}(N)$ be the set of $E\in \mathcal{E}(N)$ satisfying: 
  \begin{enumerate}
  \item $j(E)<\log |\Delta|$, where $\Delta$ is the minimal discriminant of $E$;
       \item $\frac{\Delta}{N_E}$ is squarefree;
       \item if $p\mid N_E$,  then $v_p(\Delta)\neq 3$ (equivalently, the Kodaira  Symbol of $E$ at $p$ is not III)   \footnote{Check \cite[Tabel 15.1, p. 448]{Si2009} for the  Kodaira symbols.};
      \item $E$ has good reduction at 2 and 3.
  \end{enumerate}
  By ii),  for each prime $p\mid N_E$,  we have   $v_p(\Delta)\neq 9$  \cite[Table 1]{ShSHWa2021} (equivalently,  the Kodaira Symbol of $E$ at $p$ is not III$^*$). Thus,  by applying the local density result \cite[Corollary 2.3]{ShSHWa2021} and following a proof analogous to Theorem 1 in \cite{ShSHWa2021}, one can derive an explicit asymptotic formula  of $|\mathcal{E}^{sf}_{\text{III}}(N)|$. In particular,  $ \mathcal{E}^{sf}_{\text{III}}(N)$ is a subset of $  \mathcal{E}^{add}(N)$  such that $| \mathcal{E}^{sf}_{\text{III}}(N)|\sim c_{\text{III}}N^{\frac{5}{6}}$ for some positive constant $c_{\text{III}}$.    Finally,  note that one can find $E\in \mathcal{E}^{sf}_{\text{III}}(N)$ and $p\mid N_E$ such that the Kodaira Symbol of $E$ at $p$ is II or IV. In particular, $p$ is not a prime of multiplicative reduction of $E$ and hence  $\mathcal{E}^{sf}_{\text{III}}(N)\neq \mathcal{E}^{ss}(N)$.  
\end{remark}

\begin{remark}    
     When elliptic curves are ordered by their naïve heights, Duke \cite{Du1997} proved that $c(E)=1$ for almost all elliptic curves (see also \cite{Gr2000}). 
\end{remark}

\subsection{On the number of non-surjective primes}
Assuming a positive answer to Serre's uniformity question, we expect only finitely many primes $\ell$ for which the mod $\ell$ representation of a non-CM elliptic curve is surjective. Our final result shows that for each such  non-CM elliptic curve $E/\Q$, there are few primes $\ell$, satisfying that $\ell\equiv 1\pmod{4}$ and $\ell$ is a non-surjective prime.  
\begin{theorem}\label{thm-number}
Let $E/\Q$ be a non-CM elliptic curve. Let
\[
\mathcal{L}_E:=\{\ell:   \ell>37, \ell \text{ satisfies (\ref{c0E})},   \ell\equiv 1\pmod{4}, \text{  and $\ell$ is a non-surjective prime}\}, 
\]
We have 
\[
|\mathcal{L}_E|\ll \frac{\log N_E}{\log\log N_E}.
\]
\end{theorem}
The proof of Theorem \ref{thm-number} is given in \S \ref{subsec:inertial}.  The condition \eqref{c0E} in $\mathcal{L}_E$ arises for a similar reason as that in the definition of $c_0(E)$. The condition $\ell\equiv 1   \pmod 4$ comes from our proof strategy when applying the  classification of the inertial type of elliptic curves due to Demblele, Freitas, and Voight \cite{DeFrVo2024}.

\begin{remark}
We prove the following slightly stronger from: 
\[
\prod_{\ell\in \mathcal{L}_E}\ell \ll N_E^{\frac{1+\epsilon}{4}} \text{ for any } \epsilon>0. 
\]
 Using \cite[Proposition 7.25]{Furio2025}, one obtains the weaker bound:  
\[
\prod_{\ell\in \mathcal{L}_E'}\ell \ll (N_E(\log\log N_E)^{\frac{1}{2}})^{\omega(N_E)}, 
\]
where 
    \[
\mathcal{L}_E':=\{\ell: \ell>37 \text{  and $\ell$ is a non-surjective prime}\},
\]
 and $\omega(N_E)$ denotes the number of distinct prime divisors of $N_E$.  
 
\end{remark}
\subsection{High level sketch}\label{subsec: high level sketch}

We briefly outline the strategy of the proofs for the main results  above. 

We begin with results for pairs $(E_1, E_2)$ of elliptic curves. One key ingredient is  the application of a comparison lemma (\Cref{cE1E2}),  which reduces the upper bound of $c(E_1\times E_2)$ to a Linnik type problem. \footnote{The original Linnik problem is to find the least non-quadratic residue.} More precisely, we find an upper bound for $p(E_1, E_2)$, where $p(E_1, E_2)$ is defined as the smallest prime $p$ such that $a_p(E_1)\neq\pm a_p(E_2).$ By the modularity of elliptic curves, $p(E_1,E_2)$ is related to an effective multiplicity one problem for symmetric square $L$-functions of holomorphic modular forms. When working with an arbitrary pair of elliptic curves, we can deduce Theorem \ref{unconditional for two elliptic curves thm} by comparing the Rankin-Selberg convolution of these $L$-functions. However, this result fails to match the upper bound of  $c(E_1\times E_2)$ under GRH. When working in families, we can obtain better bounds for $p(E_1, E_2)$ (see \Cref{thm. multi}) using  logarithmic derivatives of certain $L$-functions and then apply zero-density estimates discussed in \Cref{cor:final_ZDE}. This approach leads to the proof of Theorem \ref{unconditional for two elliptic curves thm family}.

For results of non-CM elliptic curves  $E$, we cannot directly apply the comparison lemma because we only have one elliptic curve a prior. To find another appropriate elliptic curve, inspired by the aforementioned work of Serre, Kraus, and Cojocaru, we choose a non-surjective prime  $\ell$ and consider the quadratic twist $E^{\varepsilon_\ell}$ of $E$, where  $\varepsilon_\ell$ is a nontrivial quadratic character defined in \eqref{epsilon}.  Establishing a better upper bound of $c(E)$ (\Cref{thm-uncod}) relies crucially on the ramification properties of  $\varepsilon_\ell$, which we record in \Cref{lemma-epsilon}. For the family version, we require a variation of the zero density estimate to first derive \Cref{comparision-thm-character} and then apply it to get  \Cref{thm-uncond-family inertia}. Note that the assumptions \eqref{inert-cond} and  \eqref{c0E}  are necessary to ensure that the quadratic characters have uniformly bounded conductors so that we can apply \Cref{comparision-thm-character}.

A crucial component of our proofs for family versions of the results above is the method of zero-density estimates, which we now briefly review.    Let $L(s, \pi)$ be an $L$-function satisfying the axioms in \cite[Section 5.1]{IwaniecKowalski2004}.  For  $\sigma\geq \frac{1}{2}$ and $T\geq 1$, we consider the quantity
\[
N_{\pi}(\sigma,T)=|\{\rho=\beta+i\gamma\colon L(\rho,\pi)=0,~\beta\geq\sigma,~|\gamma|\leq T\}|.
\]
The Generalized Riemann Hypothesis (GRH) for $L(s,\pi)$ implies that  $N_\pi(\sigma, T)=0$ for any $\sigma>\frac{1}{2}$ but we are far away from proving it. The zero-density estimates are often used as a substitute for GRH. A zero-density estimate typically refers to an upper bound for $N_\pi(\sigma, T)$ of the form:
 \begin{equation}\label{eq:family-zerodensity}
 \sum_{\pi\in \mathcal{S}} N_\pi(\sigma, T)\leq \left(|\mathcal{S}| T \right)^{A_{\mathcal{S}}(1-\sigma)+\varepsilon}
 \end{equation}
 for any $\varepsilon>0$, where $\mathcal{S}$ is a finite set of $L$-functions (satisfying the axioms in \cite[Section 5.1]{IwaniecKowalski2004}) and $A_{\mathcal{S}}$ is an absolute constant depending on $\mathcal{S}$. The proof of a zero density estimate generally requires some form of large sieve inequalities. The large sieve inequalities and zero density estimates for automorphic $L$-functions have been extensively studied (e.g., c.f. \cite{ Luo1999, DuKo2000, HumphriesThorner2024}).

Finally, we give an overview of the paper:

In \S \ref{Sec:elliptic-curve}, we introduce necessary background on Galois representations. Specifically, in \S \ref{sec:comp-lem}, we provide the comparison lemma that relates $c(E_1\times E_2)$ with $p(E_1, E_2)$. Then, in \S \ref{sec: epilon}, we introduce the quadratic character $\varepsilon_\ell$,  which is a key object in the proofs of  
\Cref{thm-uncond-family inertia}. We remark that part i) of 
 \Cref{lemma-epsilon} 
explains where  the assumptions \eqref{inert-cond} and \eqref{c0E} come from. 
Finally, in \S \ref{subsec:inertial}, we prove \Cref{thm-number} using Lemma \ref{quadratic-character-conductor} that gives refined ramification properties of the quadratic character $\varepsilon_\ell$. To prove this lemma, we require some  background on local Galois representations and  the Weil-Deligne representations of elliptic curves, which are given in \S \ref{sec:local-Galois} and \S \ref{sec:weil-deligne}, respectively.

In \S \ref{Sec:mult-one-and-density}, we will prove several Linnik type results. In \S \ref{subsec: zero density}, we will review the automorphic representations of $\GL_n(\A_{\Q})$ and establish several zero density estimates result (see \Cref{cor:final_ZDE}), based on the recent work of Humphries and Thorner  \cite[Theorem 1.1]{HumphriesThorner2024}. Subsequently, in \S \ref{subsec: explicit formula} and \S \ref{subsec: dyadic refinement}, we apply the explicit formula method and a dyadic refinement to obtain \Cref{thm. multi} and Theorem \ref{comparision-thm-character}, the desired Linnik type results.

In \S \ref{Sec:single-ellptic}, we prove unconditional upper bounds  for an arbitrary elliptic curve and  for elliptic curves in families (\Cref{thm-uncod} and \Cref{thm-uncond-family inertia}, respectively). 
In \S \ref{Sec:pair-elliptic}, we prove the corresponding results  (Theorem \ref{unconditional for two elliptic curves thm} and \Cref{unconditional for two elliptic curves thm family}, respectively) for pairs of elliptic curves. 

\subsection*{Acknowledgments} 
We would like to thank Jesse Thorner for helping us to refine the zero density estimates and the explicit formulas in \S \ref{Sec:mult-one-and-density}. We also thank Ramla Abdellatif, Chantal David, Lorenzo Furio, Andrew Granville, Jacob Mayle, Joseph Silverman, Asif Zaman, Pengcheng Zhang for helpful conversations and comments.

\section{Preliminaries}\label{Sec:elliptic-curve}

\subsection{A comparison lemma}\label{sec:comp-lem}
We keep the notation from the introduction.
Let  $E_1$ and $E_2$ be  non-CM, not $\overline{\Q}$-isogenous elliptic curves.  In this section,  we will relate $c(E_1\times E_2)$ with the quantity $p(E_1\times E_2)$ defined below, which  is  related to a Linnik type problem discussed in \S \ref{Sec:mult-one-and-density}.

We set
\begin{align*}
p(E_1, E_2):&=\min\{\textup{$p$: $p\nmid N_{E_1} N_{E_2}$ and $|a_p(E_1)|\neq |a_p(E_2)|$}\}\\
&=\max\{\textup{$\ell$: if $p\nmid N_{E_1} N_{E_2}$ and $p<\ell$, then $|a_{p}(E_1)|= |a_{p}(E_2)|$}\}.
\end{align*}
Note that the constant $p(E_1, E_2)$ exists and is finite by Faltings's isogeny theorem.

We recall from \cite[Corollary 3.8]{MaWa2023b} that $\im(\overline{\rho}_{E_1\times E_2, \ell})=\GL_2(\F_\ell)\times_{\det} \GL_2(\F_\ell)$ if and only if all three conditions hold:
\begin{enumerate}
\item  $\im(\overline{\rho}_{E_1, \ell})=\GL_2(\F_\ell)$,
\item  \label{} $\im(\overline{\rho}_{E_2, \ell})=\GL_2(\F_\ell)$, 
\item \label{trace} there is a prime $p\nmid N_{E_1}N_{E_2}$ such that 
\[
\tr(\overline{\rho}_{E_1, \ell}(\Frob_p)) \neq \pm \tr(\overline{\rho}_{E_2, \ell}(\Frob_p)).
\]
\end{enumerate}
Using this criteria, we can bound $c(E_1\times E_2)$ in terms of $c(E_1)$, $c(E_2)$, and $p(E_1, E_2)$.

\begin{lemma}\label{cE1E2}
Let $E_1/\Q$ and $E_2/\Q$ be non-CM elliptic curves that are not $\overline{\Q}$-isogenous. Then, 
\[
c(E_1\times E_2) \leq \max\{c(E_1), c(E_2), 4p(E_1,  E_2)^{\frac{1}{2}}\}.
\]
\end{lemma}
\begin{proof}
From the definition of $c(E_1)$ and $c(E_2)$, it is clear that if $\ell>\max\{c(E_1), c(E_2)\}$, then $\im(\overline{\rho}_{E_1, \ell})=\im(\overline{\rho}_{E_2, \ell})=\GL_2(\F_\ell)$.
We set $P:=p(E_1, E_2)$. Since 
\[
\tr(\overline{\rho}_{E_i, \ell}(\Frob_P)) \equiv a_P(E_i) \pmod \ell, \ 1\leq i\leq 2,
\]
if we take $\ell$ so that 
\[
\ell>\max\{|a_P(E_1)+a_P(E_2)|, |a_P(E_1)-a_P(E_2)|\},
\]
then 
$
\tr(\overline{\rho}_{E_i, \ell}(\Frob_P)) \neq \pm \tr(\overline{\rho}_{E_i, \ell}(\Frob_P)).
$
Combining with  the Hasse-Weil bound  $|a_P(E_i)|\leq 2P^{\frac{1}{2}}$, we conclude that if
$
\ell>\max\{c(E_1), c(E_2), 4P^{\frac{1}{2}}\},
$
then $\overline{\rho}_{E_1\times E_2, \ell}$ is surjective. Therefore, $c(E_1\times E_2)\leq \max\{c(E_1), c(E_2), 4P^{\frac{1}{2}}\}$.
\end{proof}


\subsection{Definition and ramification properties of $\varepsilon_\ell$}\label{sec: epilon}
 We keep the notation from the introduction.   Let $E/\Q$ be a non-CM elliptic curve and $\ell>37$ be a non-surjective prime. In this case, $\im(\bar{\rho}_{E,\ell})$ is (up to conjugation) contained in the normalizer of a non-split Cartan $\mathcal{C}_{ns}^+(\ell)$ but not properly contained in the non-split Cartan subgroup $\mathcal{C}_{ns}^+(\ell)$ itself. Based on this, Serre \cite{Se1972} considered the following quadratic Galois character $\varepsilon_\ell$ of $\Gal(\overline{\Q}/\Q)$: 
\begin{equation} \label{epsilon}
\varepsilon_{\ell}:=\varepsilon_{E, \ell} : \Gal(\overline{\Q}/{\Q}) \overset{\bar{\rho}_{E,\ell}}{\longrightarrow} G_E(\ell) \longrightarrow \frac{\mathcal{C}_{ns}^+(\ell)}{\mathcal{C}_{ns}(\ell)} \overset{\simeq}{\longrightarrow} \{\pm 1\}.
\end{equation}
We say $\varepsilon_\ell$ is {\emph{unramified}} at  $p$ if $\varepsilon_\ell$ is the identity when restricted to the inertia group $I_p$ at $p$. 
We also assume $\varepsilon_\ell$ is a  primitive quadratic character (see \cite[Section 3.3]{BiDi2014}).

 To study the ramification properties of $\varepsilon_\ell$, we also  need  the definitions of reduction types of $E$. First, we recall that the conductor $N_E$ of $E/\Q$ is defined as (see, e.g., \cite[p. 256]{Si2009}):
\begin{equation}\label{conductor}
N_{E} = \prod_p p^{v_p(N_E)}, \ v_p(N_E) = \begin{cases}
0 & p \text{ prime of good reduction of $E$}\\
1 & p \text{ prime of multiplicative reduction of $E$}\\
2 & p \text{ prime $\neq 2, 3$ and of additive reduction of $E$}\\
2+\delta_p & p = 2, 3 \text{ and of additive reduction of $E$},
\end{cases}
 \end{equation}
 where $v_p(\cdot)$ is the $p$-adic valuation  and $\delta_p\leq 6$ if $p=2$ and $\delta_p\leq 3$ if $p=3$. We denote by $N_E^{\text{mult}}$  the product of $p\neq 2, 3$ such that $p$ is a  prime of multiplicative reduction of $E$ and by $N_E^{\text{add}}$ the product of  $p\neq 2, 3$ such that $p$ is a  prime of additive reduction. Observe that
$
 N_E^{\text{add}} \leq N_E^{\frac{1}{2}}.    
$

 We now state the main result of this section. Its proof will require the use of the finite group  $\Phi_p$ of $\Gal(\ol{\Q}_p/\Q_p)$, which captures the failure of semistability of $E$ at $p.$ A detailed discussion of $\Phi_p$ will be provided in the next section.  
\begin{proposition} \label{lemma-epsilon} Let $E/\Q$ be a non-CM elliptic curve and $\ell>37$ be a non-surjective prime of $E$. Then, the following statements hold.
\begin{enumerate}

    \item \label{lemma-epsilon-unramified}  $\varepsilon_\ell$ is surjective and is ramified only at primes $p$ of additive, potentially good reduction of $E$ or $p\in \{2, 3\}$. More precisely, for each $\Frob_q\in \Gal(\ol{\Q}/\Q)$ with $q\nmid \ell N_E$,  
we have
   $\varepsilon_{\ell}(\Frob_q) =   \kronecker{D_\ell}{q}$, where  $D_\ell$ is an integer supported at primes $p$ satisfying one of:   
   \begin{itemize}
       \item $p$ is a prime of additive, potentially good reduction of $E$ such that $|\Phi_p|=4$,
       \item $p=\ell$  is a prime of additive, potentially good reduction of $E$, 
       \item $p\in \{2, 3\}$.
   \end{itemize}
    \item \label{lemma-epsilon-ap} For each prime $p \nmid N_E$ such that $\varepsilon_\ell(\Frob_p) = -1$, we have $\ell\mid a_p(E)$. 
\end{enumerate}
\end{proposition}
\begin{proof}
i). The surjectivity follows from \cite[Theorem 1.8]{FuLo2023}. Assume $p\neq \ell$ and $p\mid D_\ell$. Since the statements above are known if $E$ is semistable or if $j(E)$ is integral \cite[Lemma 2, p. 295 and \S 5.8, p. 317]{Se1972}, we only need to assume $p$ is a prime of additive, potentially multiplicative reduction. Since the proof can be performed locally, we may assume $E$ is an elliptic curve over the local field $\Q_p$. 
   
   From \cite[(b), p. 312]{Se1972}, we know that  $|\Phi_p|=2$. Hence,  each matrix in  $\ol{\rho}_{E, \ell}(I_p)\simeq \Phi_p$ has order $2.$ Note that the only possible matrix of order $2$ in $\mathcal{C}_{ns}^+(\ell)$ is $-I (\in \mathcal{C}_{ns}(\ell))$. This implies $\varepsilon_{\ell}(I_p)=1$ and $\varepsilon_\ell$ is unramified at $p$. 

   Next, assume $p=\ell$ and $p\mid D_\ell$. We only need to exclude that $p$ is a prime of  additive, potentially multiplicative reduction. This follows from \cite[Proposition 3.3]{LemosII2019}.

    ii) This is because if $\varepsilon_\ell(\Frob_p)=-1$, then $\ol{\rho}_{E, \ell}(\Frob_p)\in \mathcal{C}_{ns}^+(\ell)\backslash \mathcal{C}_{ns}(\ell)$. Hence $\tr(\ol{\rho}_{E, \ell}(\Frob_p))=0\in \F_\ell$. The result follows.  
\end{proof}


\subsection{Proof of \Cref{thm-number} }\label{subsec:inertial}
 In this section, we first prove  \Cref{quadratic-character-conductor} that gives refined ramification information of  $\varepsilon_{\ell}$. This requires some background including  the local  $\ell$-adic  Galois representations,  the Weil-Deligne representations of elliptic curves, and a  classification result on the inertial types of $E$ at primes of additive reduction of $E$. Afterwards, we will give the proof of \Cref{thm-number} using \cref{quadratic-character-conductor}. Throughout, we assume $\ell\neq p$ and retain some notation from the earlier sections. 

\subsubsection{Local Galois representations}\label{sec:local-Galois}

Let $E$ be an elliptic curve defined over the local field $\Q_p$. Assume $p$ is a prime of additive reduction of $E$. In other words,  $p$ is either a prime of  potentially good reduction or potentially multiplicative reduction of $E$ \cite[Proposition 5.4, p. 197]{Si2009}. The (local) $\ell$-adic Galois representation of $E$ is defined similarly as in the  global case:  
\[
\rho_{E, \ell}:\Gal(\ol{\Q}_p/\Q_p)\to \Aut(\varprojlim_nE[\ell^n])\simeq \GL_2(\Z_\ell),
\]
where $E[\ell^n]$ is the group of $\ell^n$-torsion points of $E$. Similarly, we can define the mod $\ell$ Galois representation $\ol{\rho}_{E, \ell}$. 
 We will recall some facts about the restriction of $\rho_{E, \ell}$  on the inertia group $I_p$ (up to conjugation) of $\Gal(\ol{\Q}_p/\Q_p)$ at  $p$.

 If $E$ is not semistable at $p$, then there is a nontrivial finite subgroup $\Phi_p$  that measures the failure of semistability of $E$ at $p$, defined as follows.  By \cite{SeTa1968}, the action of $I_p$ on $E[\ell]$  factors through a finite quotient group $\Phi_p$ of $I_p$, which produces the sequence 
\[
\ol{\rho}_{E, \ell}\mid_{I_p}:  I_p \twoheadrightarrow \Phi_p \hookrightarrow \Aut (E[\ell]).
\]
Therefore, $\Phi_p\simeq \bar{\rho}_{E,\ell}(I_p)\simeq \operatorname{Gal}(L/\Q_p^{ur})$, where $\Q_p^{ur}$ is the  maximal unramified extension of $\Q_p$ and $L$ is the  smallest extension
of $\Q_p^{ur}$, over which $E$ attains good reduction. It turns out that $\Phi_p$ can be identified as a subgroup of  $\Aut_{\ol{\F}_p}(E_p)$, where $E_p$ is the reduction of $E$ at a prime of $L$ above $p$. 

 We  now assume $p$ is a prime of additive, potentially good reduction. 
If $p\neq 2, 3$, 
then $\Phi_p$ is cyclic and  
\begin{enumerate}
\item $|\Phi_p|=2 \Leftrightarrow v_p(\Delta)= 6 \Leftrightarrow E \text{ is of type $\text{I}_0^*$};
$
\item  $|\Phi_p|=3 \Leftrightarrow v_p(\Delta)= 4, 8 \Leftrightarrow E \text{ is of type $\text{IV}$ or $\text{IV}^*$};
$
\item $|\Phi_p|=4 \Leftrightarrow v_p(\Delta)= 3, 9 
\Leftrightarrow E \text{ is of type $\text{III}$ or $\text{III}^*$},
$
\item $|\Phi_p|=6 \Leftrightarrow v(\Delta)= 2, 10 
\Leftrightarrow E \text{ is of type $\text{II}$ or $\text{II}^*$},
$
\end{enumerate}
where $v_p(\Delta)$ is the  valuation of the minimal discriminant $\Delta$ of $E$.
For $p=2, 3$, we have 
\begin{enumerate}
    \item if $p=3$, then $\Phi_p\simeq \Z/3\Z \rtimes \Z/4\Z$; 
    \item if $p=2$, then  $\Phi_p\simeq Q_8$, where $Q_8$ is the quaternion group of order 8;
    \item if $p=2$, then $\Phi_p\simeq \SL_2(\F_3)$. 
\end{enumerate}
We refer the reader to  \cite[5.6]{Se1972} or \cite{Kr1995} for more details.

\subsubsection{Weil-Deligne representations}\label{sec:weil-deligne}
Next, we  recall some results on Weil-Deligne representations of elliptic curves.  Most of the statements can be found in \cite{Se1972, DeFrVo2024}.

Let $p$ be a rational prime, $F/\Q_p$ be a finite extension, and $\ol{F}$ the algebraic closure of $F$. The Weil group $W_F$ of  $F$ is a locally
compact totally disconnected subgroup of the absolute Galois group  $G_F:=\Gal(\ol{F}/F)$. Let $\pi$ be a uniformizer of $F$, $F_\pi$ the Frobenius element of $W_F$, and $q$ be the size of the residue field. Let $I_F$ be the inertia group of $G_F$. Then, by local class field theory, for any $g\in W_F$, there is a unique $n\in \Z$ such that $g=F_\pi^n g_0$ for some $g_0\in I_F$. We define $|g|:=q^{-n}$.  Then, a two dimensional Weil-Deligne representation is a pair $(\rho, N)$ such that 
\begin{enumerate}
\item $\rho_: W_F\to \GL_2(\CC)$ is a finite-dimensional, semisimple, smooth, complex representation of  $W_F$ with open kernel; 
\item $N\in \GL_2(\CC)$ is nilpotent and satisfies 
\[
\rho(g) N\rho(g)^{-1}=|g| N \; \; \forall g \in W_F.
\]
\end{enumerate}
The classification of two dimensional Weil-Deligne representations are well-understood. Any such representation is one of the principal series representations, special (Steinberg) representations, and supercuspidal representations. 

It is often useful to study refined  classification results by examining their restrictions to inertia subgroups. 
An {\emph{inertial Weil-Deligne type}} is an equivalence class $[\rho, N]$ of $(\rho, N)$, where  the equivalence condition is given by  $(\rho, N)\sim (\rho', N')$ if and only if there is $P\in \GL_2(\CC)$ such that $\rho'(g)= P \rho(g) P^{-1}$ and $N'= P \rho(g) N^{-1}$ for all $g\in I_F$. An equivalence class $[\rho, N]$ is determined by the pair $(\tau, N)$, where $\tau=\rho\mid_{I_F}$ with $(\rho, N)$ a Weil-Deligne representation. 

 Given an elliptic curve $E/F$, where $F$ has residual characteristic $p$, we can attach  to $E$  a 
Weil-Deligne representation $(\rho_{E}, N)$,   where $\rho_E$ is induced from the $\ell$-adic Galois representation $\rho_{E, \ell}: G_F\to \GL_2(\Q_\ell)$ by fixing an embedding $\iota: \Q_\ell \to \CC$ and  
$N=0$ or $\begin{pmatrix} 0 & 1 \\ 0 & 0 \end{pmatrix}$, depending on if $p$ is a prime of  potentially good or potentially  multiplicative reduction \cite[Sections 4, 14, and 15]{Ro1994}. The Weil-Deligne representation  $(\rho_{E}, N)$ attached to $E/F$ is independent of the choice of $\ell$ and $\iota$.  We define the inertial Weil-Deligne type $\tau_E$ of $E$ as the equivalence class of $[\rho_{E}, N]$. 

We assume $E/F$ has  additive, potentially good reduction.
The following classification of the inertial Weil-Deligne type  $\tau_E$  holds, where we use some terminologies and notation in \cite[Section 2.5]{DeFrVo2024}. 
\begin{lemma}[Proposition 4.2.1\cite{DeFrVo2024}]\label{Phi_p=4}
 Let $p\geq 5$ be a prime. Let $E/\Q_p$ be an elliptic curve of additive, potentially good reduction such that $|\Phi_p|=4$. Let $\tau_E$ be the inertial Weil-Deligne type of $E$. 
\begin{enumerate}
\item If $p\equiv 1\pmod 4$, then $\tau_E$ is a principal series representation.
\item If $p\equiv 3\pmod 4$, then  $\tau_E$ is a nonexceptional supercuspidal representation.
\end{enumerate}
Moreover explicitly,
\begin{enumerate}
    \item If $p\equiv 1\pmod 4$,  then $\tau_E \simeq \tau_{ps}(1,1,4)$. In this case, 
    $\tau_E=\chi\mid_{I_p}\oplus   \chi^{-1}\mid_{I_p}$, where $\chi: W_{\Q_p}\to \mathbb{C}^\times$ is a cyclotomic character of order 4.
    \item If $p\equiv 3\pmod 4$, then $\tau_E \simeq \tau_{sc}(u,2,4)$, where $u\in \Z_p^\times$ is a nonsquare. In this case,
    $\tau_E=\chi\mid_{I_K}\oplus   \chi^{-1}\mid_{I_K}$, where $K=\Q_p(\sqrt{u})$ and $\chi: W_{K}\to \mathbb{C}^\times$ is a cyclotomic character of order $4.$
\end{enumerate}

\end{lemma}

\subsubsection{The proofs}
 The following result is a refinement of \Cref{lemma-epsilon}, because it gives a if and only if condition for the ramified primes of $\varepsilon_\ell$.  
\begin{lemma}\label{quadratic-character-conductor}
    Let $E/\Q$ be a non-CM  elliptic  curve. Let $\ell>37$ be a prime such that $\ol{\rho}_{E, \ell}$ is nonsurjective and $\ell\equiv 1\pmod 4$. Then $\varepsilon_{ \ell}$ is ramified at a prime $p\neq \ell$  if and only if  $p$ is a prime  of additive,  potentially good reduction with  $|\Phi_p|=4$ or  $p\in \{2, 3\}$.
\end{lemma}
\begin{proof}
The only if direction follow from part i) of \Cref{lemma-epsilon}. 
For the other direction,  we take $p\notin \{ 2, 3\}$ and assume $p$ is a prime of additive,  potentially good reduction  with  $|\Phi_p|=4$. Hence,   $\ol{\rho}_{E, \ell}(I_p)$ is isomorphic to a cyclic group of order 4 in $\GL_2(\F_\ell)$. 

Recall that by fixing a nonsquare element $\varepsilon \in \F_\ell^\times$, we can write, up to conjugation,   
\[
\mathcal{C}_{ns}(\ell) = \left\{\begin{pmatrix} a & \varepsilon c \\ c & a\end{pmatrix} : a,c \in \F_\ell \text{ and } (a,c) \neq (0,0)\right\},
\; \; 
\mathcal{C}^+_{ns}(\ell) = \mathcal{C}_{ns}(\ell) \cup \begin{pmatrix} 1 & 0 \\ 0 & -1 \end{pmatrix} \mathcal{C}_{ns}(\ell).
\]
Note that $\varepsilon_{\ell}$ is unramified at $p$ if and only if  $\varepsilon_{ \ell}$ is trivial when restricted to $I_p$, which is equivalent to say that
$
\ol{\rho}_{E, \ell}(I_p)\subseteq \mathcal{C}_{ns}(\ell) \simeq \F_{\ell^2}^{\times}.
$
Assume  for the sake of contradiction, $\ol{\rho}_{E, \ell}(I_p)$ is generated by 
\[
g=\begin{pmatrix} a & \varepsilon c \\ c & a\end{pmatrix} \in \mathcal{C}_{ns}(\ell).
\]
Using the fact that $g$ has order 4, a straightforward computation  implies $a\neq 0$ (otherwise, we would get $\varepsilon c^2=-1$, which is impossible as $\ell\equiv 1 \pmod 4$ and $\varepsilon$ is a nonsquare). In particular, \[
\tr \ g \not \equiv  0\pmod \ell.
\]
On the other hand, we see that 
$
\tr ( \rho_{E, \ell}(\sigma))= 0
$ for any  $\sigma \in I_p$. This is because by Lemma \ref{Phi_p=4},  $\chi(\sigma)+\chi(\sigma)^{-1}=i+(-i)=0$, where $i=\sqrt{-1}\in \Q_\ell\subseteq \CC$. \footnote{Since $\ell\equiv 1\pmod 4$, we can view $\pm i$ sa elements in $\Z_\ell$ upon fixing an  embedding $\Q_\ell \hookrightarrow \CC$.} Hence,  
$
\tr(\ol{\rho}_{E, \ell}(\sigma))\equiv 0\pmod \ell
$
and we reach a contradiction.
Therefore, $\ol{\rho}_{E, \ell}(I_p)\not\subseteq \mathcal{C}_{ns}(\ell)$ and $\varepsilon_{\ell}$ is ramified at $p$.

Now, assume $p\in \{2,  3\}$ and $p$ a prime of additive, potentially good reduction.   We recall  facts from \S \ref{sec:local-Galois} that the groups  $\Phi_2$ and $\Phi_3$ are not abelian. Since $\mathcal{C}_{ns}(\ell)$ is abelian,  the image of $\ol{\rho}_{E, \ell}(I_p)$ can not be contained in $\mathcal{C}_{ns}(\ell)$. So $\varepsilon_{ \ell}$ is always ramified at 2 and 3.
This completes the proof.
\end{proof}

\begin{remark}
We need to assume that $p\neq \ell$ because,  when $p=\ell$ and $p$ is a prime of additive, potentially good reduction of $E$, much less is known about the the group $\Phi_p$. This lack of understanding makes it hard to analyze the ramification of  $\varepsilon_\ell$ at $p$.
\end{remark}

We now end this section with the proof of Theorem \ref{thm-number}.   
\begin{proof}[Proof of Theorem \ref{thm-number}]
For each  $\ell\in \mathcal{L}_E$,  we can associate a primitive quadratic  character $\varepsilon_{\ell}$  for the absolute Galois group $\Gal(\ol{\Q}/\Q)$.   By \Cref{lemma-epsilon} i), we have
\[
\varepsilon_{\ell}(\cdot)=\left( \frac{2^{v_2}3^{v_3}\ell^{v_\ell}D}{\cdot}\right), \text{ where $\displaystyle D=\prod_{\substack{p\mid N_E^{\add}\\p\neq 2, 3, \ell\\
|\Phi_p|=4}} p$, $0\leq v_2 \leq 3$, and  $v_3, v_\ell\in \{0, 1\}$}.
\]   
Further, by part i) of \Cref{lemma-epsilon} and \Cref{quadratic-character-conductor} and the definition of $\mathcal{L}_E$,  we obtain that the character $\varepsilon_{\ell}$ is independent of the choice of $\ell$ whenever $\ell\equiv 1\pmod 4$ and $\ell$ is not a prime of additive, potentially good reduction. 
Hence, from Theorem \ref{thm-uncod}, there is a prime $p\nmid N_E$ such that  $p\ll N_E^{\frac{1+\varepsilon}{2}}$ and  
\[
\ell\mid a_p(E) \quad \text{ for all} \quad \ell\in \mathcal{L}_E.
\]
Then by  part ii) of \Cref{lemma-epsilon}, we get 
\[
\prod_{\ell\in \mathcal{L}_E} \ell \mid  |a_p(E)|\leq 2\sqrt{p}\ll N_E^{\frac{1+\varepsilon}{4}}.
\]
Finally, taking the logarithm on both sides and applying the partial summation formula, we obtain the desired bound. 
\end{proof}

\section{Zero density estimates and Linnik type results}\label{Sec:mult-one-and-density}

 In this section, we study  two Linnik-type problems: determining the least prime  $p$ such that $\lambda_p(f)\neq \pm \lambda_p(g)$, where   $f$ and $g$  are holomorphic newforms of fixed weight, and identifying the smallest prime $p$ such that $\lambda_p(f)\neq \chi(p)\lambda_p(f)$, where  $f$ is a holomorphic newform and $\chi$ is a quadratic character. These results will later be applied to the case of non-CM elliptic curves  over $\Q$ that are not $\overline{\Q}$-isogenous. Our main tools are zero density estimates and the explicit formula. 

\subsection{Zero density estimates}\label{subsec: zero density}
Let $\mathbb{A}_{\mathbb{Q}}$ be the ring of ad\`{e}les over $\mathbb{Q}$.  Let $\mathfrak{F}_{n}$ be the set of cuspidal automorphic representations $\pi=\bigotimes_{v} \pi_{v}$ of $\mathrm{GL}_{n}(\mathbb{A}_{\mathbb{Q}})$ with trivial central character. Let $q_{\pi}$ be the arithmetic conductor of $\pi$, $C(\pi)\geq 1$ the analytic conductor of $\pi$, and $\mathfrak{F}_n(Q)=\{\pi\in\mathfrak{F}_n\colon C(\pi)\leq Q\}$. The analytic conductor $C(\pi)$ is a useful measure for the arithmetic and spectral complexity of $\pi$. Let $\mathbbm{1}\in\mathfrak{F}_1$ be the trivial representation, whose $L$-function is the Riemann zeta function $\zeta(s)$.

Let $\pi$ be a  cuspidal automorphic representation of $\GL_n(\mathbb{A}_{\mathbb{Q}}).$ We can associate a degree $n$
$L$-function:
\[L(s,\pi)=\sum_{n\geq1}\frac{\lambda_{\pi}(n)}{n^s}=\prod_{p<\infty}\prod_{i=1}^{n}\left(1-\frac{\alpha_{p,i}}{p^s}\right)^{-1}.\]
which is absolutely convergent for $\Re(s)>1$. We call the conjugate class (of diagonal matrices) $[\diag(\alpha_{p,1},\ldots, \alpha_{p,n})]$ the Satake parameters at $p.$ In particular, when $p\nmid q_{\pi}$, $\alpha_{p,i}\neq0.$ The well-known Ramanuajn conjecture predicts that $|\alpha_{p,i}|= 1$ for all $1 \leq i \leq n.$ This was proved, by Deligne, when $\pi$ corresponds to a holomorphic newform in $\mathfrak{F}_2$.

\begin{remark}
Our normalization for the central characters ensures that $|\mathfrak{F}_n(Q)|$ is finite, and that $L(s,\pi)$ satisfies a functional equation relating to $L(1-s,\widetilde{\pi})$. 
\end{remark}

Let $\pi'$ be another  cuspidal automorphic representation of $\GL_{n'}(\mathbb{A}_{\mathbb{Q}})$. We can define the Rankin-Selberg $L$-function $L(s,\pi\times\pi').$ More precisely, suppose that for $p\nmid q_{\pi}$ (resp. $p\nmid q_{\pi'}$), the Satake parameter of $\pi$ (resp. $\pi'$) is $[\diag(\alpha_{p,1},\ldots,\alpha_{p,n})]$ (resp.  $[\diag(\beta_{p,1},\ldots,\beta_{p,n'})]$), then
\[L(s,\pi\times\pi')=\prod_{p<\infty}L_p(s,\pi_p\times \pi_{p}'),\]
and when $p\nmid q_{\pi}q_{\pi'}$, we have
\[L_p(s,\pi_p\times \pi_{p}')=\prod_{1\leq i\leq n}\prod_{1\leq j\leq n'}\left(1-\frac{\alpha_{p,i}\beta_{p,j}}{p^s}\right)^{-1}.\]
We also define the von Mangoldt function $\Lambda_{\pi\times\pi'}$ associated with $\pi\times\pi'$  via the logarithmic differentiation:
\[-\frac{L'(s,\pi\times\pi')}{L(s,\pi\times\pi')}=\sum_{n=1}^{\infty}\frac{\Lambda_{\pi\times\pi'}(n)}{n^s},\]
provided that $\re(s)>1.$ A direct calculation will show: $\Lambda_{\pi\times\pi'}$ is only supported on prime powers and  when $p\nmid q_{\pi}q_{\pi'}$,
\[\Lambda_{\pi\times\pi'}(p^{\ell})=\log p\sum_{1\leq i\leq n}\sum_{1\leq j\leq n'}\alpha_{p,i}^{\ell}\beta_{p, j}^{\ell}.\]
This implies $\Lambda_{\pi\times\pi'}(p)=\lambda_{\pi}(p)\lambda_{\pi'}(p)\log p.$

Define, for $\sigma\geq 0$ and $T\geq 1$, the quantity
\[
N_{\pi\times\pi'}(\sigma,T)=|\{\rho=\beta+i\gamma\colon L(\rho,\pi\times\pi')=0,~\beta\geq\sigma,~|\gamma|\leq T\}|.
\]
Note that $N_{\pi\times\pi'}(\frac{1}{2},T)$ is roughly $T\log(C(\pi)C(\pi')T)$ via the argument principle and the functional equation, and GRH can be restated as $N_{\pi\times\pi'}(\sigma,T)=0$ for any $\sigma>\frac{1}{2}$.

\begin{theorem}
	\label{thm:ZDE}
	Let $n,n'\geq 1$ and $Q,T\geq 1$.  Let $\mathcal{S}\subseteq\mathfrak{F}_n$, $\mathcal{S}(Q)=\{\pi\in\mathcal{S}\colon C(\pi)\leq Q\}$, and $\pi'\in\mathfrak{F}_{n'}$.  If $\varepsilon>0$ and $\sigma\geq 0$, then
	\[
	\sum_{\pi\in\mathcal{S}(Q)}N_{\pi\times\pi'}(\sigma,T)\ll_{n,n'} \big(|\mathcal{S}(Q)|^4 (C(\pi')QT)^{6.15\max\{n^2,n'n\}})^{1-\sigma+\varepsilon}.
	\]
	In particular, if $n'=1$ and $\pi'=\mathbbm{1}$, then
	\[
	\sum_{\pi\in\mathcal{S}(Q)}N_{\pi}(\sigma,T)\ll_{n} \big(|\mathcal{S}(Q)|^4 (QT)^{6.15n^2})^{1-\sigma+\varepsilon}.
	\]
\end{theorem}
\begin{proof}
This is a specialization of \cite[Theorem 1.1]{HumphriesThorner2024}.
\end{proof}

Let $\mathscr{F}$ be a set of normalized non-dihedral holomorphic newforms of fixed weight. For each $f\in \mathscr{F},$ we can associate an automrophic cuspidal representation $\pi_f\in\mathfrak{F}_2.$ 
Let $\mathscr{F}(Q)=\{f\in\mathscr{F}\colon C(f)\leq Q\}$. For any $f\in\mathscr{F}$, the $L$-function is 
\[L(s,f)=\sum_{n\geq1}\frac{\lambda_f(n)}{n^s}=\prod_{p<\infty}\left(1-\frac{\alpha_{f,p}}{p^s}\right)^{-1}\left(1-\frac{\beta_{f,p}}{p^s}\right)^{-1}.\]
When $p\nmid q_f$, $\alpha_{f,p}\beta_{f,p}=1$ and $|\alpha_{f,p}|=|\beta_{f,p}|=1.$ For any $j\geq1$, we can define the $n$-th symmetric $L$-function $L(s,\mathrm{Sym}^nf)$ via a Euler product:
\[L(s,\mathrm{Sym}^nf)=\prod_{p<\infty}L_p(s,\mathrm{Sym}^nf)\]
satisfying, when $p\nmid q_f$,
\[L_p(s,\mathrm{Sym}^nf)=\prod_{\ell=0}^{n}\left(1-\frac{\alpha_{f,p}^{n-2\ell}}{p^s}\right)^{-1}.\]

Notice that each modular form  $f\in\mathscr{F}$ corresponds to a cuspidal automorphic representation $\pi_f$ of $\GL_2(\A_{\Q}).$ A recent  result of Newton and Thorne \cite{NeTh2021} implies  that the $n$-th symmetric power lift  $\mathrm{Sym}^nf:=\mathrm{Sym}^n\pi_f$ is a  cuspidal automorphic representation of $\GL_{n+1}(\mathbb{A}_{\mathbb{Q}}).$
\begin{theorem}
If $f\in\mathscr{F}$ and $n\in\{2,4\}$, then $\mathrm{Sym}^nf$ lies in $\mathfrak{F}_{n+1}$ and has trivial central character.  
\end{theorem}
\begin{proof}
If $n=2$ (resp. $n=4$), this was implied by \cite[Theorem (2.2)]{GJ} (resp. \cite[Theorem B]{Kim}). For general $n,$ this was proved in \cite{NeTh2021}.
\end{proof}

\begin{lemma}
\label{lem:Ramakrishnan}
Let $f,g\in\mathscr{F}$.  If $\mathrm{Sym}^2f=\mathrm{Sym}^2g$ or $\mathrm{Sym}^4f=\mathrm{Sym}^4g$, then there exists a primitive Dirichlet character $\chi$ of order dividing $2$ such that $\pi_g=\pi_f\otimes\chi$.
\end{lemma}
\begin{proof}
The result for the symmetric square lift follows from \cite[Theorem 4.1.2]{Ramakrishnan}.  To handle the symmetric fourth power lift, we observe via \cite[Lemma 4.2]{JT_Siegel} that if $\mathrm{Sym}^4f=\mathrm{Sym}^4g$, then there exists a primitive Dirichlet character $\chi$ such that $\pi_g=\pi_f\otimes\chi$.  Since $g\in\mathscr{F}$, we have that $g$ is self-dual with trivial central character.  It follows that $\chi$ must have order dividing $2$.
\end{proof}

For $j\in\{2,4\}$, we define an equivalence relation $\sim_j$ on $\mathscr{F}(Q)$ by
\[
\textup{$f\sim_jg$ if and only if $\mathrm{Sym}^{j}f=\mathrm{Sym}^{j}g$.}
\]
Let $\widetilde{\mathscr{F}}_j=\mathscr{F}(Q)/\sim_j$ be the set of equivalence classes.  Since each equivalence class is finite, we can choose a representative having the largest analytic conductor; let $\mathscr{F}_j$ be a set of such representatives.  We have the bounds $|\mathscr{F}_j|\leq |\mathscr{F}(Q)|$ and $C(\mathrm{Sym}^2f)\ll C(f)^2\asymp  q_{f}^2$.\footnote{In our case, the analytic conductor is $\mathfrak{q}(f)$ in \cite[Page 95]{IwaniecKowalski2004}. When the weight of the modular form $f\in\mathscr{F}$ is fixed,  $c(f)\asymp q_f,$ the level of $f.$}  In addition, we have that $C(\mathrm{Sym}^4f)\ll C(f)^5\asymp q_{f}^{5}$ \cite[Theorem A.1]{HIJT}.  Therefore, it follows from \Cref{thm:ZDE} that
\begin{equation}
\label{eqn:zde_pre}
\begin{aligned}
\sum_{f\in\mathscr{F}_2}N_{\mathrm{Sym}^2f}(\sigma,Q)&\ll_{\varepsilon} \big(|\mathscr{F}(Q)|^{4} (Q^{2} \cdot Q)^{6.15\cdot (3)^2}\big)^{1-\sigma+\varepsilon},\\
\sum_{f\in\mathscr{F}_4}N_{\mathrm{Sym}^4f}(\sigma,Q)&\ll_{\varepsilon} \big(|\mathscr{F}(Q)|^{4} (Q^{5} \cdot Q)^{6.15\cdot (5)^2}\big)^{1-\sigma+\varepsilon},\\
\sum_{f\in\mathscr{F}_2}N_{\mathrm{Sym}^2f\times\mathrm{Sym}^2f_0}(\sigma,Q)&\ll_{\varepsilon} \big(|\mathscr{F}(Q)|^{4} (Q^2\cdot Q^{2}\cdot Q)^{6.15\cdot (3)^2}\big)^{1-\sigma+\varepsilon},\qquad \forall f_0\in\mathscr{F}(Q).
\end{aligned}
\end{equation}

We conclude the following result.
\begin{corollary}
\label{cor:final_ZDE}
Fix $\delta>0$ and let $\varepsilon>0$.  If $\sigma\geq 0$, $Q\geq 1$, $|\mathscr{F}(Q)|\gg Q^{\delta}$, and $\chi$ is a primitive Dirichlet character of absolutely bounded conductor, then
\begin{align*}
\sum_{f\in\mathscr{F}(Q)}N_{\mathrm{Sym}^2f\otimes\chi}(\sigma,Q)&\ll_{\varepsilon} |\mathscr{F}(Q)|^{\frac{1}{2\delta}+(4+\frac{166.05}{\delta})(1-\sigma)+\frac{\varepsilon}{5}},\\
\sum_{f\in\mathscr{F}(Q)}N_{\mathrm{Sym}^4f}(\sigma,Q)&\ll_{\varepsilon} |\mathscr{F}(Q)|^{\frac{1}{2\delta}+(4+\frac{922.5}{\delta})(1-\sigma)+\frac{\varepsilon}{5}},\\
\sum_{f\in\mathscr{F}(Q)}N_{\mathrm{Sym}^2f\times\mathrm{Sym}^2f_0}(\sigma,Q)&\ll_{\varepsilon} |\mathscr{F}(Q)|^{\frac{1}{2\delta}+(4+\frac{276.75}{\delta})(1-\sigma)+\frac{\varepsilon}{5}},\qquad \forall f_0\in\mathscr{F}(Q).
\end{align*}
If the conductor of each $\pi\in \mathscr{F}(Q)$ is squarefree, then the $1/(2\delta)$ in the exponent can be removed.
\end{corollary}
\begin{proof}
Proceeding as in \cite[p.~11]{DuKo2000}, we find that if $j\in\{2,4\}$, then each equivalence class in $\widetilde{\mathscr{F}}_j$ has cardinality $O_{\varepsilon}(Q^{1/2+\varepsilon})$.  Therefore, since $Q\ll |\mathscr{F}(Q)|^{1/\delta}$, the desired result follows from  \eqref{eqn:zde_pre}.  The last claim follows from the fact that if $f,g\in\mathscr{F}(Q)$, there exists a primitive Dirichlet character $\chi$ such that $\pi_g=\pi_f\otimes\chi$ and $q_{f}$ and $q_{g}$ are both squarefree, then $\chi$ is trivial \cite[Theorem A]{RamakrishnanYang}.
\end{proof}

\subsection{Identification and the explicit formula}\label{subsec: explicit formula}

We keep the notation from the earlier sections.

For $\delta>\frac{1}{2}$, we define
\begin{equation}\label{eq. constant c}
    c(\delta)=\frac{2(4+\frac{923}{\delta})}{1-\frac{1}{2\delta}}.
\end{equation}
Let $Q\geq 1000$ and recall $\mathscr{F}(Q)$ is a set of   normalized holomorphic newforms  with conductor at most $Q$ and for any $f\in\mathscr{{F}}(Q),$ $q_f$ is its level. Fix $f_0\in \mathscr{F}(Q)$ and we define
\[
M(f_0;Q,\varepsilon)=\{f\in\mathscr{F}(Q)\colon \textup{if $p\leq (\log Q)^{c(\delta)+\varepsilon}$ and $p\nmid q_{f}q_{f_0}$, then $|\lambda_{f}(p)|=|\lambda_{f_0}(p)|$}\},
\]

We want to show that for certain  choices \footnote{Here ``certain" means $L(s,\mathrm{Sym}^2f_0)$ and $L(s,\mathrm{Sym}^4f_0)$ have a relatively large zero free region  (see \Cref{thm:multiplicity_one_1}).} of $f_0$ and any $0<\varepsilon<1$, $|M(f_0;Q,\varepsilon)|$ is  small compared to $|\mathscr{F}(Q)|$. If $f\in M(f_0; Q, \varepsilon)$ with $f\neq f_0$, then
\[
|\lambda_{f}(p)|^2 = |\lambda_{f_0}(p)|^2,\qquad p\leq (\log Q)^{c(\delta)+\varepsilon},\qquad p\nmid q_{f}q_{f_0}.
\]
Since $|\lambda_{f}(p)|^2=1+\lambda_{\mathrm{Sym}^2f}(p)$ (and similarly for $f_0$), it follows that
\[
 \lambda_{\mathrm{Sym}^2f}(p) =\lambda_{\mathrm{Sym}^2f_0}(p),\qquad p\leq (\log Q)^{c(\delta)+\varepsilon},\qquad p\nmid q_{f}q_{f_0}.
\]
In particular,
\[
 \lambda_{\mathrm{Sym}^2f_0}(p)\lambda_{\mathrm{Sym}^2f_0}(p)=\lambda_{\mathrm{Sym}^2f}(p)\lambda_{\mathrm{Sym}^2f_0}(p),\qquad p\leq (\log Q)^{c(\delta)+\varepsilon},\qquad p\nmid q_{f}q_{f_0}.
\]
Since $p\nmid q_{f}q_{f_0}$,  we conclude that 
\[
1+\lambda_{\mathrm{Sym}^2f_0}(p)+\lambda_{\mathrm{Sym}^4f_0}(p)= \lambda_{\mathrm{Sym}^2f\times \mathrm{Sym}^2f_0}(p),\qquad p\leq (\log Q)^{c(\delta)+\varepsilon},\quad p\nmid q_{f}q_{f_0}.
\]

It follows that if $\phi\colon\mathbb{R}\to\mathbb{R}$ is an infinitely differentiable function supported in $[1/2,1]$ with $\widehat{\phi}(1)=1$ and 
\begin{equation}\label{eq: x-logQ}
    x=(\log Q)^{c(\delta)+\varepsilon},
\end{equation}
then
\[
\sum_{p\nmid q_{f}q_{f_0}}(1+\lambda_{\mathrm{Sym}^2f_0}(p)+\lambda_{\mathrm{Sym}^4f_0}(p))\phi(p/x)=\sum_{p\nmid q_{f}q_{f_0}}\lambda_{\mathrm{Sym}^2f\times\mathrm{Sym}^2f_0}(p)\phi(p/x).
\]
Trivially incorporate ramified primes and all prime powers using progress towards Ramanujan:
\[
\sum_{n=1}^{\infty}(\Lambda(n)+\Lambda_{\mathrm{Sym}^2f_0}(n)+\Lambda_{\mathrm{Sym}^4f_0}(n))\phi(n/x)=\sum_{n=1}^{\infty}\Lambda_{\mathrm{Sym}^2f\times\mathrm{Sym}^2f_0}(n)\phi(n/x)+O(\sqrt{x}).
\]
Then the Mellin inversion formula implies:
\begin{align*}
&\frac{1}{2\pi i}\int_{2-i\infty}^{2+i\infty}\Big(-\frac{\zeta'}{\zeta}(s)-\frac{L'}{L}(s,\mathrm{Sym}^2f_0)-\frac{L'}{L}(s,\mathrm{Sym}^4f_0)\Big)x^s\widehat{\phi}(s)ds\\
&=\frac{1}{2\pi i}\int_{2-i\infty}^{2+i\infty}-\frac{L'}{L}(s,\mathrm{Sym}^2f\times\mathrm{Sym}^2f_0)x^s\widehat{\phi}(s)ds+O(\sqrt{x}).
\end{align*}
Push the contour and we obtain:
\begin{align*}
&x - \sum_{\substack{\rho=\beta+i\gamma \\ \zeta(\rho)=0}}x^{\rho}\widehat{\phi}(\rho) - \sum_{\substack{\rho=\beta+i\gamma \\ L(\rho,\mathrm{Sym}^2f_0)=0}}x^{\rho}\widehat{\phi}(\rho)- \sum_{\substack{\rho=\beta+i\gamma \\ L(\rho,\mathrm{Sym}^4f_0)=0}}x^{\rho}\widehat{\phi}(\rho)\\
&=-\sum_{\substack{\rho=\beta+i\gamma \\ L(\rho,\mathrm{Sym}^2f\times\mathrm{Sym}^2f_0)=0}}x^{\rho}\widehat{\phi}(\rho)+O(\sqrt{x}).
\end{align*}
 The contribution from the trivial zeros is negligible, and upon reordering the summations, we arrive at:
\begin{align*}
&x - \sum_{\substack{\rho=\beta+i\gamma \\ \zeta(\rho)=0}}x^{\rho}\widehat{\phi}(\rho)\\
&=\sum_{\substack{\rho=\beta+i\gamma \\ L(\rho,\mathrm{Sym}^2f_0)=0}}x^{\rho}\widehat{\phi}(\rho) + \sum_{\substack{\rho=\beta+i\gamma \\ L(\rho,\mathrm{Sym}^4f_0)=0}}x^{\rho}\widehat{\phi}(\rho) -\sum_{\substack{\rho=\beta+i\gamma \\ L(\rho,\mathrm{Sym}^2f\times\mathrm{Sym}^2f_0)=0}}x^{\rho}\widehat{\phi}(\rho)+O(\sqrt{x}).
\end{align*}
Then apply the Prime number theorem:
\begin{equation}
\label{eqn:PNT_classical}
x\asymp \Big|\sum_{\substack{\rho=\beta+i\gamma \\ L(\rho,\mathrm{Sym}^2f_0)=0}}x^{\rho}\widehat{\phi}(\rho) + \sum_{\substack{\rho=\beta+i\gamma \\ L(\rho,\mathrm{Sym}^4f_0)=0}}x^{\rho}\widehat{\phi}(\rho) -\sum_{\substack{\rho=\beta+i\gamma \\ L(\rho,\mathrm{Sym}^2f\times\mathrm{Sym}^2f_0)=0}}x^{\rho}\widehat{\phi}(\rho)\Big|.
\end{equation}
Let $T\geq 2$.  By triangle inequality, we trivially estimate zeros with $|\gamma|>T$:
\begin{align*}
&\sum_{\substack{\rho=\beta+i\gamma \\ L(\rho,\mathrm{Sym}^2f_0)=0}}x^{\rho}\widehat{\phi}(\rho) + \sum_{\substack{\rho=\beta+i\gamma \\ L(\rho,\mathrm{Sym}^4f_0)=0}}x^{\rho}\widehat{\phi}(\rho) -\sum_{\substack{\rho=\beta+i\gamma \\ L(\rho,\mathrm{Sym}^2f\times\mathrm{Sym}^2f_0)=0}}x^{\rho}\widehat{\phi}(\rho)\\
&\ll \sum_{\substack{\rho=\beta+i\gamma \\ L(\rho,\mathrm{Sym}^2f_0)=0 \\ |\gamma|\leq T}}\frac{x^{\beta}}{|\gamma|+1} + \sum_{\substack{\rho=\beta+i\gamma \\ L(\rho,\mathrm{Sym}^4f_0)=0\\ |\gamma|\leq T}}\frac{x^{\beta}}{|\gamma|+1}+\sum_{\substack{\rho=\beta+i\gamma \\ L(\rho,\mathrm{Sym}^2f\times\mathrm{Sym}^2f_0)=0\\ |\gamma|\leq T}}\frac{x^{\beta}}{|\gamma|+1}+\frac{x\log(Q T)}{T^2}.
\end{align*}

Fix  $\varepsilon$ such that $0<\varepsilon<\frac{1}{2\delta}$.   Let
\[
T = Q,\qquad \sigma = 1-\frac{1-\frac{1}{2\delta}-\varepsilon}{4+\frac{922.5}{\delta}},\qquad \mathcal{R}(\sigma,T)=\{s\in\mathbb{C}\colon \mathop{\mathrm{Re}}(s)\geq \sigma,~|\mathop{\mathrm{Im}}(s)|\leq T\}.
\]

\begin{proposition}\label{thm:multiplicity_one_1}
Fix $\delta>0$. Let $Q\geq 5$ and assume that $|\mathscr{F}(Q)|\gg Q^{\delta}$.  Let $f_0\in\mathscr{F}(Q)$ satisfying that $L(s,\mathrm{Sym}^2f_0)L(s,\mathrm{Sym}^4f_0)$ has no zero in $\mathcal{R}(\sigma,T)$. Then,  there is a small $\varepsilon>0$ (depending on $\delta$) such that
\[
|M(f_0;Q,\varepsilon)|\ll_{\varepsilon}\begin{cases}
|\mathscr{F}(Q)|^{1-\frac{1}{2\delta}-\varepsilon}&\mbox{if each $f \in\mathscr{F}(Q)$ has a squarefree conductor,}\\
|\mathscr{F}(Q)|^{1-\varepsilon}&\mbox{otherwise.}
\end{cases}
\]  
\end{proposition}

\begin{proof}
By our assumption, $L(s,\mathrm{Sym}^2f_0)L(s,\mathrm{Sym}^4f_0)$ has no zero in $\mathcal{R}(\sigma,T)$. Since $L(s,\mathrm{Sym}^2f_0)$ (resp. $L(s,\mathrm{Sym}^4f_0)$) have $O(\log(Q(|t|+3)))$ nontrivial zeros $\beta+i\gamma$ such that $|\gamma-t|\leq 1$ (e.g., \cite[Proposition 5.7(1)]{IwaniecKowalski2004}), it follows that
\begin{align*}
&\sum_{\substack{\rho=\beta+i\gamma \\ L(\rho,\mathrm{Sym}^2f_0)=0 \\ |\gamma|\leq T}}\frac{x^{\beta}}{|\gamma|+1} + \sum_{\substack{\rho=\beta+i\gamma \\ L(\rho,\mathrm{Sym}^4f_0)=0\\ |\gamma|\leq T}}\frac{x^{\beta}}{|\gamma|+1}+\frac{x\log(Q T)}{T^2}\\
&\ll x^{\sigma}\int_1^T \frac{\log(Q(t+3))}{t+1}dt+\frac{x\log(QT)}{T^2}\ll x^{\sigma+\frac{2}{c(\delta)+\varepsilon}}=o(x).
\end{align*}

Our hypothesis that $Q$ is large ensures that the implied constant is independent of $\varepsilon$. 
  
  Similarly,  if $f\in M(f_0;Q,\varepsilon)-\{f_0\}$ and $L(s,\mathrm{Sym}^2f\times\mathrm{Sym}^2f_0)$ has no zero in $\mathcal{R}(\sigma,T)$, then
\[
\sum_{\substack{\rho=\beta+i\gamma \\ L(\rho,\mathrm{Sym}^2f\times\mathrm{Sym}^2f_0)=0\\ |\gamma|\leq T}}\frac{x^{\beta}}{|\gamma|+1}\ll x^{\sigma}\int_1^T \frac{\log(Q(t+3))}{t+1}dt+\frac{x\log(QT)}{T^2}\ll x^{\sigma+\frac{2}{c(\delta)+\varepsilon}}=o(x).
\]
Therefore, by \eqref{eqn:PNT_classical}, we have $x=o(x)$ as $Q$ (hence $x$) grows, a contradiction.  Therefore, every $f\in M(f_0;Q,\varepsilon)-\{f_0\}$ has the property that $L(s,\mathrm{Sym}^2f\times\mathrm{Sym}^2f_0)$ {\it must} have a zero in $\mathcal{R}(\sigma,T)$.  The number of such $f$ is bounded by the last part of  \Cref{cor:final_ZDE}. 
\end{proof}

To prove Theorem \ref{thm-uncond-family inertia}, we will need the following \Cref{prop.comparision-thm-character}, a variant of Proposition \ref{thm:multiplicity_one_1}. Since the proof follow closely along the same lines, we only outline the key steps here. 
\begin{proposition}\label{prop.comparision-thm-character}
Fix $\delta>0$.    Let $Q\geq 5$ and assume that $|\mathscr{F}(Q)|\gg Q^{\delta}$.   For  a primitive quadratic character  $\chi\pmod{q}$  of conductor $q\ll 1$ and any $\epsilon>0$, denote by 
\[
M(\chi;Q,\varepsilon):=\{f\in\mathscr{F}(Q)\colon \textup{if $p\leq (\log Q)^{c(\delta)+\varepsilon}$ and $p\nmid q_{f}q$, then $\lambda_{f}(p)=\lambda_{f}(p)\chi(p)$}\}.
\]
Then, there is a  (small) $\varepsilon>0$ (depending on $\delta$) such that 
\begin{align*}
|M(\chi;Q,\varepsilon)|\ll_{\varepsilon}\begin{cases}
|\mathscr{F}(Q)|^{1-\frac{1}{2\delta}-\varepsilon}&\mbox{if each $f\in\mathscr{F}(Q)$ has squarefree conductor,}\\ 
|\mathscr{F}(Q)|^{1-\varepsilon}&\mbox{otherwise.}
\end{cases}
\end{align*}
The first case is nontrivial when $\delta>\frac{1}{2}$.
\end{proposition}

{\bf Sketch:} Let $\phi\colon\mathbb{R}\to\mathbb{R}$ be  an infinitely differentiable function supported in $[1/2,1]$ with $\widehat{\phi}(1)=1$.  If $f\in M(\chi;Q,\varepsilon)$, then   $\lambda_{f}(p)=\lambda_{f}(p)\chi(p)$ for all $p\leq x = (\log Q)^{c(\delta)+\varepsilon}$, hence we have
\[
\sum_{n=1}^{\infty}\Lambda_{f\times f}(n)\phi(n/x)=\sum_{n=1}^{\infty}\Lambda_{f\times f}(n)\chi(n)\phi(n/x)+O(\sqrt{x}).
\]

By Mellin inversion, pushing the contour and estimating the contribution from the trivial zeros trivially,
\[
x - \sum_{\substack{\rho = \beta+i\gamma \\ \zeta(\rho)=0}}x^{\rho}\widehat{\phi}(\rho)-\sum_{\substack{\rho = \beta+i\gamma \\ L(\rho,\mathrm{Sym}^2f)=0}}x^{\rho}\widehat{\phi}(\rho)=- \sum_{\substack{\rho = \beta+i\gamma \\ L(s,\chi)=0}}x^{\rho}\widehat{\phi}(\rho)-\sum_{\substack{\rho = \beta+i\gamma \\ L(\rho,\mathrm{Sym}^2f\otimes\chi)=0}}x^{\rho}\widehat{\phi}(\rho)+O(\sqrt{x}).
\]
Prime number theorem for arithmetic progressions (since the conductor of $\chi$ is $O(1)$):
\begin{equation}
\label{eqn:PNTAP_classical}
x\asymp\Big|\sum_{\substack{\rho = \beta+i\gamma \\ L(\rho,\mathrm{Sym}^2f)=0}}x^{\rho}\widehat{\phi}(\rho)-\sum_{\substack{\rho = \beta+i\gamma \\ L(\rho,\mathrm{Sym}^2f\otimes\chi)=0}}x^{\rho}\widehat{\phi}(\rho)\Big|.
\end{equation} 

By trivially estimating zeros with $|\gamma|>T$,
\begin{align*}
&\sum_{\substack{\rho = \beta+i\gamma \\ L(\rho,\mathrm{Sym}^2f)=0}}x^{\rho}\widehat{\phi}(\rho)-\sum_{\substack{\rho = \beta+i\gamma \\ L(\rho,\mathrm{Sym}^2f\otimes\chi)=0}}x^{\rho}\widehat{\phi}(\rho)\\
&\ll \sum_{\substack{\rho = \beta+i\gamma \\ |\gamma|\leq T \\ L(\rho,\mathrm{Sym}^2f)=0}}\frac{x^{\beta}}{1+|\gamma|}+\sum_{\substack{\rho = \beta+i\gamma \\ |\gamma|\leq T\\  L(\rho,\mathrm{Sym}^2f\otimes\chi)=0}}\frac{x^{\beta}}{1+|\gamma|}+\frac{x\log(QT)}{T^2}.
\end{align*}
Let
\[
T = Q,\qquad \sigma = 1-\frac{1-\frac{1}{2\delta}-\varepsilon}{4+\frac{922.5}{\delta}},\qquad \mathcal{R}(\sigma,T)=\{s\in\mathbb{C}\colon \mathop{\mathrm{Re}}(s)\geq \sigma,~|\mathop{\mathrm{Im}}(s)|\leq T\}.
\]
Assume $f\in M(\chi;Q,\varepsilon)$ is taken such that $L(s,\mathrm{Sym}^2f)L(s,\mathrm{Sym}^2f\times\chi)$ has no zero in $\mathcal{R}(\sigma,T)$. Since $L(s,\mathrm{Sym}^2f)$ (resp. $L(s,\mathrm{Sym}^2f\times\chi)$) have $O(\log(Q(|t|+3)))$ nontrivial zeros $\beta+i\gamma$ such that $|\gamma-t|\leq 1$ \cite[Proposition 5.7(1)]{IwaniecKowalski2004}. It follows that
\begin{align*}
&\sum_{\substack{\rho=\beta+i\gamma \\ L(\rho,\mathrm{Sym}^2f)=0 \\ |\gamma|\leq T}}\frac{x^{\beta}}{|\gamma|+1} +\frac{x\log(Q T)}{T^2}\ll x^{\sigma}\int_1^T \frac{\log(Q(t+3))}{t+1}dt+\frac{x\log(QT)}{T^2}=o(x)
\end{align*}
and
\[
\sum_{\substack{\rho = \beta+i\gamma \\ |\gamma|\leq T\\  L(\rho,\mathrm{Sym}^2f \otimes\chi)=0}}\frac{x^{\beta}}{1+|\gamma|}\ll x^{\sigma}\int_1^T \frac{\log(Q(t+3))}{t+1}dt=o(x).
\]
 Therefore, by \eqref{eqn:PNTAP_classical}, we have $x=o(x)$ as $Q$ (hence $x$) grows, a contradiction.  
Therefore, $f\in M(\chi;Q,\varepsilon) $ implies that either $L(s,\mathrm{Sym}^2f)$ or $L(s,\mathrm{Sym}^2f\times\chi)$  has a zero in $\mathcal{R}(\sigma,T).$ Then $|M(\chi;Q,\varepsilon)|$ is bounded  by the first two parts of Corollary \ref{cor:final_ZDE}. This completes the proof of  \Cref{prop.comparision-thm-character}.
\qed

\subsection{Linnik type results}\label{subsec: dyadic refinement}
We keep the notation from \S \ref{subsec: zero density}. 
In this section, we will replace the $\log Q$ in $M(f_0,Q,\varepsilon)$ (resp. $M(\chi;Q,\varepsilon)$) by $\log q_f.$ We will use a dyadic division trick as follows. Fix $f_0$ satisfying conditions in Proposition \ref{thm:multiplicity_one_1}. It follows from \Cref{thm:multiplicity_one_1} that if $Q\geq 3$, then
\begin{align*}
&\#\{f\in\mathscr{F}(Q)-\mathscr{F}(Q/2)\colon \textup{if $p\leq (\max\{\log C(f),\log C(f_0)\})^{c(\delta)+\varepsilon}$ and $p\nmid q_{f}q_{f_0}$, then $|\lambda_{f}(p)|=|\lambda_{f_0}(p)|$}\}\\
&\ll_{\varepsilon}\begin{cases}
|\mathscr{F}(Q)|^{1-\frac{1}{2\delta}-\varepsilon}&\mbox{if each $f \in\mathscr{F}(Q)$ has squarefree conductor,}\\
|\mathscr{F}(Q)|^{1-\varepsilon}&\mbox{otherwise.}
\end{cases}
\end{align*}

We decompose $[1000,Q]$ into $O(\log Q)$ dyadic subintervals, applying the above bound to each of them.  Since $\log Q\leq \kappa^{-1}Q^{\kappa}$ for all $Q\geq 1$ and $\kappa>0$, we conclude (upon 
rescaling $\varepsilon$ and choosing $\kappa$ in terms of $\varepsilon$) the following theorem.
\begin{theorem}\label{thm. multi}
Fix $\delta>0$.   Let $Q\geq 5$ and $\mathscr{F}$ be a set of newforms associated to non-CM elliptic curves over $\Q$, satisfying $|\mathscr{F}(Q)|\gg Q^{\delta}$. Fix $f_0\in\mathscr{F}(Q)$ and assume that $L(s,\mathrm{Sym}^2f_0)L(s,\mathrm{Sym}^4f_0)$ has no zero in $\mathcal{R}(\sigma,T)$. 
For any $\epsilon>0$, set
\[
\mathcal{M}(f_0;Q,\varepsilon)=\{f\in\mathscr{F}(Q)\colon \textup{if $p\leq (\max(\log q_f,\log q_{f_0}))^{c(\delta)+\varepsilon}$ and $p\nmid q_{f}q_{f_0}$, then $|\lambda_{f}(p)|=|\lambda_{f_0}(p)|$}\}.
\]
Then, there is a  small $\varepsilon>0$ (depending on $\delta$) such that 
\begin{align*}
\mathcal{M}(f_0;Q,\varepsilon)\ll_{\varepsilon}\begin{cases}
|\mathscr{F}(Q)|^{1-\frac{1}{2\delta}-\varepsilon}&\mbox{if each $\pi\in\mathscr{F}(Q)$ has squarefree conductor,}\\
|\mathscr{F}(Q)|^{1-\varepsilon}&\mbox{otherwise.}
\end{cases}
\end{align*}
The first case is nontrivial when $\delta>\frac{1}{2}$.
\end{theorem}
To prove Theorem \ref{thm. multi}, we also use the fact that, for $f$ of fixed weight, $\log C(f)=\log q_f+O(1).$

A similar argument  shows the following.
\begin{theorem}
\label{comparision-thm-character}
Fix $\delta>0$.    Let $Q\geq 5$ and assume that $|\mathscr{F}(Q)|\gg Q^{\delta}$. For  any primitive quadratic character  $\chi\pmod{q}$   of conductor $q\ll 1$ and $\varepsilon>0$, set
\[
\mathcal{M}(\chi;Q,\varepsilon)=\{f\in\mathscr{F}(Q)\colon \textup{if $p\leq (\log q_f)^{c(\delta)+\varepsilon}$ and $p\nmid q_{f}q$, then $\lambda_{f}(p)=\lambda_{f}(p)\chi(p)$}\}.
\]
Then, there is a small $\varepsilon>0$ (depending on $\delta$) such that 
\begin{align*}
|\mathcal{M}(\chi;Q,\varepsilon)|\ll_{\varepsilon}\begin{cases}
|\mathscr{F}(Q)|^{1-\frac{1}{2\delta}-\varepsilon}&\mbox{if each $f\in\mathscr{F}(Q)$ has squarefree conductor,}\\
|\mathscr{F}(Q)|^{1-\varepsilon}&\mbox{otherwise.}
\end{cases}
\end{align*}
The first case is nontrivial when $\delta>\frac{1}{2}$.
\end{theorem}

\begin{remark}
In \Cref{thm. multi} (resp. \Cref{comparision-thm-character}), if we  assume GRH for $L(s,\mathrm{Sym}^2f\times\mathrm{Sym}^2f_0)$ and $L(s,\mathrm{Sym}^2f_0\times\mathrm{Sym}^2f_0)$ (resp. $L(s,\mathrm{Sym}^2f)$ and $L(s,\mathrm{Sym}^2f \otimes \chi)$), then there is a prime $p\ll (\max\{\log q_f, \log q_{f_0}\})^2$ (resp. $p\ll (\log q_f)^2$) such that $|\lambda_f(p)|\neq |\lambda_{f_0}(p)|$ (resp. $\lambda_f(p)\neq \chi(p) \lambda_{f}(p)$) \cite[Proposition 5.22]{IwaniecKowalski2004}.
\end{remark}

\section{Results for  elliptic curves}\label{Sec:single-ellptic}

We keep the notation form \S \ref{Sec:elliptic-curve} and \S \ref{Sec:mult-one-and-density}. 
We also specialize $\mathscr{F}$ to be a set of non-dihedral holomorphic cusp forms of weight $2$ associated to non-CM elliptic curves over $\Q$ by the modularity theorem. 
Note that  if we assume that an elliptic curve $E/\Q$ corresponds to a newform $f_E,$ then $a_p(E)=\lambda_{f_E}(p)p^{1/2}.$ 
We also recall that the arithmetic conductor $q_{f_E}$ of $f_E$ is exactly the conductor   $N_E$ of $E$.

\subsection{Proof of Theorem \ref{thm-uncod}}
\begin{proof}
Let $\ell> 37$ be a prime such that $\ell$ is a non-surjective prime. If this happens, then from \S \ref{sec: epilon}, we obtain a natural character $\varepsilon_\ell$ and the quadratic twist $E^{\varepsilon_\ell}$. Since $E$ is non-CM, $E^{\varepsilon_\ell}$ and $E$ are not isogenous over $\Q$.  
By part (2) of  Lemma \ref{lemma-epsilon} and \cite[Theorem 7.4]{Iwaniec1997} or \cite[Proposition 2.2]{Co2005},  primes of   multiplicative reduction  of $E$ remains to be so for   $E^{\varepsilon_{\ell}}$, while the exponent of additive primes of $E^{\varepsilon_{\ell}}$ can only decrease. 
Therefore, we obtain
\[ N_{E^{\varepsilon_\ell}} \mid 6^{100}N_E.  
\]
We denote $f_E$ and $f_{E^{\epsilon_\ell}}$ the weight 2 newforms associated to $E$ and $E^{\varepsilon_\ell}$ by the modularity theorem.
Applying  \cite[Lemma 11]{Mu1999}, 
we conclude that there is a positive integer $n$, coprime to $N_E$ such that  $a_n(f)\neq a_n(f')$ and 
\[
n \ll_{\varepsilon} N_E^{1+\varepsilon}. 
\]
In particular,  we can find a prime $p\mid n$ such that
$a_p(f_E)\neq a_p(f_{E^{\epsilon_\ell}})$ and $\varepsilon_\ell(\Frob_p)=-1$ (otherwise, we would have $a_p(f_E)= a_p(f_{E^{\epsilon_\ell}})$). Finally, we derive from part (3) of Lemma \ref{lemma-epsilon} that
\[
\ell \mid a_p(E) \ll 2\sqrt{p}\ll_{\varepsilon} N_E^{\frac{1+\varepsilon}{2}}.
\]

\end{proof}

\subsection{Proof of \Cref{thm-uncond-family inertia}}
We recall the notation from \S \ref{Sec:mult-one-and-density}.
\begin{proof}
Fix a sufficiently small  $\varepsilon>0$ and let $\delta$ be the constant in \Cref{thm-uncond-family inertia}. As the result holds for arbitrary small $\varepsilon>0$,  we will, unless otherwise necessary, treat any constant multiple of $\varepsilon$ as $\varepsilon$.
It suffices to show 
\begin{equation}\label{inert-excpetional-set}
   \# \{E\in \mathcal{E}'(N): c_0(E)>  (\log N_E)^{c(\delta)/2+\varepsilon}\}=O(|\mathcal{E}'(N)|^{1-\varepsilon}),
\end{equation}
where $c(\delta)$ is given by \eqref{eq. constant c}. 

By part i) of  \eqref{epsilon}, 
if $\ell$ is a prime such that  $\ell> 37$, is a non-surjective prime of $E\in \mathcal{E}'(N)$, and satisfies (\ref{c0E}), then  $\varepsilon_{\ell}(\cdot)=\chi_{a, b}(\cdot) :=\left(\frac{2^a3^b}{\cdot}\right)$ for some  positive integers $a, b$ such that $0\leq a, b\leq 3$ and $(a, b)\neq (0, 0)$. In particular, the quadratic character is independent of the choice of $E$ and $\ell$. 

For any $E$ counted by (\ref{inert-excpetional-set}), denote by  $\ell=c_0(E)$.  Let $p_0:=p(E)\nmid 6 N_E$ be the least prime such that 
\[
\varepsilon_{\ell}(\Frob_{p_0})=\chi_{a, b}(p_0)=-1 \; \text{ and } \; a_{p_0}(E)\neq 0
\]
or equivalently,
\[
\chi_{a, b}(p_0)a_{p_0}(E)\neq a_{p_0}(E).
\]
Then, by par ii) of Lemma \ref{lemma-epsilon}, we obtain 
\[
c_0(E)\mid a_{p_0}(E) \Longrightarrow (\log N_E)^{c(\delta)/2+\epsilon}<c_0(E) \leq |a_{p_0}(E)| \leq 2\sqrt{p_0}.
\]
Hence $p_0>\frac{1}{4}(\log N_E)^{c(\delta)+\varepsilon}$. Therefore, for any $p\leq \frac{1}{4}(\log N_E)^{c(\delta)+\varepsilon}$ and  $p\nmid N_E$, we have $\chi_{a, b}(p)a_p(E)=a_p(E)$.

Now for each pair $(a, b)$ with $0\leq a, b\leq 3$ and $(a, b)\neq (0, 0)$, we consider the family 
\[
\mathcal{F}^{a, b}(N):=\left\{E\in \mathcal{E}'(N): \text{ if $ p\leq \frac{1}{4}(\log N_E)^{c(\delta)+\varepsilon}$ and $p\nmid 6N_E$, then }
\chi_{a, b}(p)a_p(E)=a_p(E) 
\right\}
\]
From the discussion above,
\[
 \{E\in \mathcal{E}'(N): c_0(E)>  (\log N_E)^{c(\delta)/2+\varepsilon}\}\subseteq \bigcup_{\substack{0\leq a,b\leq 3\\ (a, b)\neq (0, 0)}}\mathcal{F}^{a, b}(N).
\]
Hence, it suffices to prove 
$
 \left|\mathcal{F}^{a, b}(N)\right|\ll_\varepsilon |\mathcal{E}'(N)|^{1-\varepsilon}.
$
Since $\mathcal{F}^{a, b}(N)$ is contained in  
\[
\mathcal{M}(\chi_{a, b};N,\varepsilon)=\{E \in\mathcal{E}'(N)\colon \textup{if $p\leq (\log N_{E})^{
c(\delta)+\varepsilon}$ and $p\nmid 6N_E$, then $a_p(E)=a_p(E)\chi_{a, b}(p)$}\}, 
\]
by applying Theorem \ref{comparision-thm-character} to the set $\mathcal{M}(\chi_{a, b};N,\varepsilon)$, we obtain 
\[
|\mathcal{F}^{a, b}(N)|\leq |\mathcal{M}(\chi_{a, b}; N, \epsilon)|\ll_{\varepsilon} |\mathcal{E}'(N)|^{1-\varepsilon}.
\]
This completes the proof of the theorem.
\end{proof}

\section{Results for products of elliptic curves}\label{Sec:pair-elliptic}

As before, we keep the notation form \S \ref{Sec:elliptic-curve}  and \S \ref{Sec:mult-one-and-density} 
and assume $\mathscr{F}\subset \mathfrak{F}_2$ is a set of non-dihedral newforms of weight $2$ associated to non-CM elliptic curves over $\Q$. We denote by $f_{E}$ the weight $2$ holomorphic newform corresponding to $E/\Q$.


\subsection{An unconditional result for a product of elliptic curves}\label{uncond-2}
In this section, we give an alternative proof for a power bound of $c(E_1\times E_2)$  for  a pair of elliptic curves $(E_1, E_2)$, which was first proved by Murty \cite{Mu1999}.

\begin{theorem}
\label{unconditional for two elliptic curves thm}
    Let $E_1/\Q$ and $E_2/\Q$ be two non-CM elliptic curves such that   $E_1$ and $E_2$ are not $\ol{\Q}$-isogenous.  If $\varepsilon>0$, then $c(E_1\times E_2)\ll_{\varepsilon} (\max\{N_{E_1},N_{E_2}\})^{3+\varepsilon}$.
\end{theorem}

For each $i\in \{1, 2\}$, if $\re(s)>1,$ we write 
\[L(s,\mathrm{Sym}^2f_{E_i})=\sum_{n\geq1}\frac{\lambda_{i}(n)}{n^s}.\]
Then let $\psi$ be a smooth function compactly supported on $[1,2]$ and $X>0.$ Assume that $N_{E_1}\leq N_{E_2}.$ We consider the following sums:
\[S(X)=\sum_{\substack{n\geq1,\\(n,N_{E_1}N_{E_2})=1}}\lambda_1(n)^2\psi(n/X)\]
and
\[H(X)=\sum_{\substack{n\geq1,\\(n,N_{E_1}N_{E_2})=1}}\lambda_1(n)\lambda_2(n)\psi(n/X).\]

Denote by $\widehat{\psi}(s)$ the Mellin transformation of $\psi(x).$ Then we can find $G(s)$ such that
\[S(X)=\frac{1}{2\pi i}\int_{(2)}\widehat{\psi}(s)G(s)L(s,\mathrm{Sym}^2f_{E_1}\times\mathrm{Sym}^2f_{E_1})X^s\,ds.\]

Here $G(s)$ satisfies: 
\begin{enumerate}[(a)]
\item
$G(s)$ is defined by a Euler product and absolutely convergent when $\re(s)>1/2.$ This gives $G(s)\ll_{\varepsilon} 1$ when $\re(s)\geq \frac{1}{2}+\varepsilon$. 
\item
When $\re(s)\geq\frac{3}{4}$, we have $G(s)\gg1$. 
\end{enumerate}
Shifting the integration line to $\re(s)=\frac{3}{4}$, we obtain:
\[S(X)=G(1)XL(1,\mathrm{Sym}^2f_{E_1})L(1,\mathrm{Sym}^4f_{E_1})+O(X^{3/4}(N_{\mathrm{Sym}^2f_{E_1}}N_{\mathrm{Sym}^4f_{E_1}})^{\frac{1}{2}(1-3/4)}(\log N_{E_1})^{3+\varepsilon}).\]
This is due to the following facts: 
\begin{enumerate}[(a)]
\item  for any holomorphic cusp forms $f$ with analytic conductor $C(f),$
\[L(1,\mathrm{Sym}^nf)\ll (\log C(f))^{n+1}\]
\item for any cuspidal automorphic representation $\pi$ (with analytic conductor $C(\pi)$) satisfying Generalized Ramanujan Conjecture (see \cite{Soundararajan2010}),
\[
L(1/2,\pi)\ll\frac{C(\pi)^{\frac{1}{4}}}{(\log C(\pi))^{1-\varepsilon}}.\]
\end{enumerate}

Then apply the Phragm\'en–Lindel\"of principle and one obtains the bound  
\[L(3/4+it,\mathrm{Sym}^2f_{E_1}\times\mathrm{Sym}^2f_{E_1})\ll (1+|t|)^{A}(N_{\mathrm{Sym}^2f_{E_1}}N_{\mathrm{Sym}^4f_{E_1}})^{\frac{1}{2}(1-3/4)}(\log N_{E_2})^{3+\varepsilon}).\]
Then the error term comes from the rapid decay of $\widehat{\psi}.$

On the other hand, one can show (e.g., c.f.\cite[Lemma 4.2]{CogdellMichel2004}, \cite[Corollary 1.1]{Zhao2023}):
\[L(1,\mathrm{Sym}^nf)\gg \frac{1}{(\log C(f))^{2n+2+\varepsilon}}.\]

This will imply
\[S(X)\gg \frac{X}{(\log N_{E_1})^{16+\varepsilon}}\]
provided that $X\gg (N_{{E_1}}N_{{E_2}})^{3}$ (notice that $N_{E_1}\leq N_{E_2}.$)

We use a similar way to analyze $H(X)$ and we obtain:
\[H(X)=O(X^{3/4}N(\mathrm{Sym}^2f_{E_1}\otimes\mathrm{Sym}^2f_{E_2})^{\frac{1}{2}(1-3/4)}(\log N_{E_2})^4),\]
where $N(\mathrm{Sym}^2f_{E_1}\otimes\mathrm{Sym}^2f_{E_2})$ is the conductor of 
$\mathrm{Sym}^2f_{E_1}\otimes\mathrm{Sym}^2f_{E_2}$.

Assuming we can find $X>0$ such that   $X\gg N(\mathrm{Sym}^2f_{E_1}\otimes\mathrm{Sym}^2f_{E_2})^{\frac{1}{2}}(\log N_{E_1}N_{E_2})^{80+\varepsilon}$ and $\lambda_1(n)=\lambda_2(n)$ for all $n\leq 2X$, we would get  $S(X)=H(X)$. Hence
\[ \frac{X}{(\log N_{E_1})^{16+\varepsilon}}\ll X^{3/4}N(\mathrm{Sym}^2f_{E_1}\otimes\mathrm{Sym}^2f_{E_2})^{\frac{1}{2}(1-3/4)}(\log N_{E_2})^{4}.\]
This will force
\[X\ll N(\mathrm{Sym}^2f_{E_1}\otimes\mathrm{Sym}^2f_{E_2})^{\frac{1}{2}}(\log N_{E_1}N_{E_2})^{80+\varepsilon}.\]
A contradiction. Therefore, there is a sufficiently large constant $\Cl[abcon]{RS1}$ such that $\lambda_1(n)\neq \lambda_2(n)$ for some $n\leq \Cr{RS1} N(\mathrm{Sym}^2f_{E_1}\otimes\mathrm{Sym}^2f_{E_2})^{\frac{1}{2}}(\log N_{E_1}N_{E_2})^{80+\varepsilon}$.  Furthermore, we can assume that $(n,N_{E_1}N_{E_2})=1$ by the definition of $S(X),H(X)$. As a result, we can find a prime $p$ satisfying $(p,N_{E_1}N_{E_2})=1$ and $p\ll  N(\mathrm{Sym}^2f_{E_1}\otimes\mathrm{Sym}^2f_{E_2})^{\frac{1}{2}}(\log N_{E_1}N_{E_2})^{80+\varepsilon}$ such that $\lambda_1(p)\neq \lambda_2(p).$ This implies $\lambda_{f_{E_1}}(p)\neq\pm\lambda_{f_{E_2}}(p)$ and hence 
\[p(E_1,E_2)\ll N(\mathrm{Sym}^2f_{E_1}\otimes\mathrm{Sym}^2f_{E_2})^{\frac{1}{2}}(\log N_{E_1}N_{E_2})^{80+\varepsilon}.\]
By \Cref{thm-uncod} and \Cref{cE1E2}, we obtain
\[c(E_1\times E_2)\ll \max\{c(E_1),c(E_2), 4p(E_1,E_2)^{1/2}\}\ll  N(\mathrm{Sym}^2f_{E_1}\otimes\mathrm{Sym}^2f_{E_2})^{\frac{1}{4}}(\log N_{E_1}N_{E_2})^{40+\varepsilon}.\]
Finally, using
  \cite[Theorem 1]{BuHe1997}, we get 
  \[
  c(E_1\times E_2) \ll (N_{E_1}N_{E_2})^{\frac{3}{2}+\varepsilon}.
  \] 
This complete the proof of Theorem \ref{unconditional for two elliptic curves thm}.
\qed

\subsection{Proof of Theorem \ref{unconditional for two elliptic curves thm family}}\label{sec:unconditional for two elliptic curves thm family}

\begin{proof}
Fix a sufficiently small $\varepsilon>0$ and let $\delta$ be the constant in \Cref{unconditional for two elliptic curves thm family}. As the result holds for arbitrary small $\varepsilon>0$,  we will, unless otherwise necessary, treat any expression $\kappa\varepsilon$ as $\varepsilon$.  We denote by $\mathcal{E}_{zero}'(N)$ the subset of $\mathcal{E}'(N)$ such that  that  any  $E \not\in \mathcal{E}_{zero}'(N)$ satisfies  the property that both $L(s, \mathrm{Sym}^2(f_{E}))$ and $L(s, \mathrm{Sym}^4(f_{E}))$ have  no zero in $\mathcal{R}(\sigma, T)$ with $\sigma = 1-\frac{1-\frac{1}{2\delta}-\varepsilon}{4+\frac{922.5}{\delta}}$ and $T=N$. By \Cref{cor:final_ZDE}, 
\begin{equation}\label{eq:complem-set}
    |\mathcal{E}_{zero}'(N)|\ll  |\mathcal{E}'(N)|^{\frac{1}{2\delta}+(4+\frac{922.5}{\delta})(1-\sigma)+\frac{\varepsilon}{5}}\ll_\varepsilon |\mathcal{E}'(N)|^{1-\varepsilon}.
\end{equation}

For each $E_1\notin \mathcal{E}_{zero}'(N)$, denote by 
\begin{align*}
\mathcal{M}(f_{E_1};N,\varepsilon) & :=\\
& \hspace{-15pt} \{E_2 \in\mathcal{E}'(N)\colon \textup{$|a_p(E_2)|=|a_p(E_1)|$ for all  $p\nmid N_{E_1}N_{E_2}$ and $p\leq \max\{\log N_{E_1}, \log N_{E_2}\}^{c(\delta)+\varepsilon}$}\}.
\end{align*}
Note that   for each  pair  $(E_1, E_2)$ such that $E_1\notin \mathcal{E}_{zero}'(N)$ and $E_2\notin \mathcal{M}(f_{E_1};N,\varepsilon)$, we have 
\[
p(E_1, E_2)\leq (\max\{\log N_{E_1}, \log N_{E_2}\})^{c(\delta)+\varepsilon}. 
\]
Consequently, by  Lemma \ref{cE1E2}, we  obtain
\[
c(E_1\times E_2)\leq \max\{c(E_1), c(E_2), 4(\max\{\log N_{E_1}, \log N_{E_2}\})^{c(\delta)/2+\varepsilon}\}.
\]
Applying Theorem \ref{thm. multi} to the set $\mathcal{M}(f_{E_1};N,\varepsilon)$ and invoking \eqref{eq:complem-set}, we obtain that the cardinality of following complementary set 
\[\{(E_1,E_2)\in \mathcal{E}'(N)\times \mathcal{E}'(N): \mbox{ $E_1\in \mathcal{E}_{zero}'(N)$ or $E_2\in \mathcal{M}(f_{E_1};N,\varepsilon)$ }\}\]
is at most
\[
\left|\left(\bigcup_{E_1\in \mathcal{E}'_{zero}(N)} \mathcal{E}'(N)\right) \bigcup \left(\bigcup_{E_1\notin \mathcal{E}'_{zero}(N)} \mathcal{M}(f_{E_1};N,\varepsilon)\right) \right| \ll_\varepsilon |\mathcal{E}'(N)|\cdot |\mathcal{E}'(N)|^{1-\frac{4}{5}\varepsilon}
=o(|\mathcal{E}'(N)|^{2}).
\]
The first limit result follows from this. 

If $\mathcal{E}'(N)\subseteq \mathcal{E}^{ss}(N)$,  we recall from \eqref{eq:family-size}  that  $|\mathcal{E}'(N)|\gg N^{5/6}$ and $c(E)\leq 11$ for each $E\in \mathcal{E}^{ss}(N)$. The second limit result now follows.
\end{proof}




\begin{appendix}\label{app:app}

\section{The number of CM elliptic curves  ordered by conductor}
In this appendix, we will give an upper bound for the number of $\Q$-isogenous classes of CM elliptic curves over $\Q$, when ordered by conductor.\footnote{For the number of CM elliptic curves ordered by naive height, see \cite[Theorem 1.3]{BaCa2025}.} We were unable to find a proof in the literature. Therefore, we include it in the appendix. We will keep the notation from earlier sections.

\begin{theorem}\label{app:thm}
  Let $\mathcal{E}^{CM}(N)$ be the set of $\Q$-isogeny classes of CM elliptic curves over $\Q$ with conductor bounded by $N$. Then, 
  \[
  |\mathcal{E}^{CM}(N)|\ll N^{\frac{1}{2}}.
  \]
\end{theorem}
\begin{proof}
We first recall that  there are only 13  CM $j$-invariants taking values in $\Q$. For each such rational number $j$, we fix a representative $E$ for the $\ol{\Q}$-isomorphism class corresponding to $j$ such that $E$ has the minimal conductor. We denote by $E_0$ (resp. $E_{1728}$) for the representative with $j$-invariant 0 (resp. 1728), and by $E_1, \ldots, E_{11}$ for the rest of the $j$-invariants.
Now by \cite[III.10 Theorem 10.1  and X.2.  Theorem 2.2]{Si2009}, any $\Q$-isomorphism class of  CM elliptic curves over $\Q$ is in one of the following sets:
\begin{align}\label{eq: cm}
  & \{E_i^{\chi}: \ \chi: \Gal(\ol{\Q}/\Q)\to \Z/2\Z,  \text{ $\chi$ is primitive}, 1\leq i\leq 11\}; \\ 
& \{E_{1728}^{\chi}:  \chi: \Gal(\ol{\Q}/\Q)\to \Z/4\Z, \text{ $\chi$ is primitive}\}; \\ 
  & \{E_{0}^{ \chi}:  \chi: \Gal(\ol{\Q}/\Q)\to \Z/6\Z, \text{ $\chi$ is primitive}\}. 
\end{align}
Here, for each elliptic curve  $E$ and each Galois character  $\chi\in \Hom(\Gal(\ol{\Q}/\Q),\Aut(E))$, there is a unique twist $E^\chi \in H^1(\Gal(\ol{\Q}/\Q), \Aut(E))$  which  is defined explicitly by \cite[X.5 Proposition 5.4]{Si2009}. We say the character $\chi$ is primitive if $\chi$ factor through a cyclic Galois extension that gives rise  to a primitive Dirichlet character.

Since any $\Q$-isogeny class in  $\mathcal{E}^{CM}(N)$  contains  at most $8$ $\Q$-isomorphism classes of elliptic curves \cite{Ke1982}, in order to bound $|\mathcal{E}^{CM}(N)|$, it suffices to bound the number of $\Q$-isomorphism classes of  CM elliptic curves over $\Q$ listed by (A.1)-(A.3). The original counting problem  amounts to counting the number of  primitive characters that lie in one of the sets:
\begin{align}
    & \{\chi: \Gal(\ol{\Q}/\Q)\to \Z/2\Z \; :\: \chi \text{ is primitive, } N_{E_i^{\chi}} \leq N,  1\leq i\leq 11\};\\
   & \{\chi: \Gal(\ol{\Q}/\Q)\to \Z/4\Z \; :\: \chi \text{ is primitive, } N_{E_{1728}^{ \chi}}\leq N \}; \\
  & \{\chi: \Gal(\ol{\Q}/\Q)\to \Z/6\Z \; :\: \chi \text{ is primitive, } N_{E_{0}^{ \chi}}\leq N \}.
\end{align}

Finally, by identifying the above elliptic curves $E$ (resp. $E^\chi$) with a newform $f$ (resp. $f^\chi$) of weight 2 and applying \cite[Proposition 14.20]{DuKo2000}, we obtain 
\[
N_{E_i^{\chi}}=N_{E_i}D_\chi^2 \quad  \text{ if $(D_\chi, N_{E_i})=1$},\]
where $D_\chi$ is the conductor of $\chi$. Therefore, if $(D_\chi, N_{E_i})= 1$, then   
\begin{equation}\label{app:condu-bound}
   N_{E_i^{\chi}}\leq N \Longrightarrow  D_\chi^2 \ll N^{\frac{1}{2}}, \quad \text{ for each } i\in \{0, 1728, 1, \ldots, 11\}. 
\end{equation}
If $(D_\chi, N_{E_i})\neq 1$, then since the set  $\{E_i\}_{\substack{1\leq i\leq 11,\\ i=0, 1728}}$ is finite and $\chi$ is primitive, the bound \eqref{app:condu-bound} still holds. 
Therefore, these sets  (A.4)-(A.6) are  all bounded by $O(N^{\frac{1}{2}})$.
This gives the desired bound of $ |\mathcal{E}^{CM}(N)|$.

\end{proof}
\end{appendix}
\bibliographystyle{abbrv}
\bibliography{EOIT.bib}

\def\cprime{$'$}
\begin{thebibliography}{10}

\bibitem{MR3961086}
J.~Balakrishnan, N.~Dogra, J.~S. M\"{u}ller, J.~Tuitman, and J.~Vonk.
\newblock Explicit {C}habauty-{K}im for the split {C}artan modular curve of level 13.
\newblock {\em Ann. of Math. (2)}, 189(3):885--944, 2019.

\bibitem{BaCa2025}
A.~Barquero-Sanchez and J.~Calvo-Monge.
\newblock The density and distribution of {C}{M} elliptic curves over $\mathbb{Q}$.
\newblock {\em Journal of Mathematical Analysis and Applications}, 546(1):129192, 2025.

\bibitem{BiDi2014}
N.~Billerey and L.~V. Dieulefait.
\newblock Explicit large image theorems for modular forms.
\newblock {\em J. Lond. Math. Soc. (2)}, 89(2):499--523, 2014.

\bibitem{BiPaRe2013}
Y.~Bilu, P.~Parent, and M.~Rebolledo.
\newblock Rational points on {$X^+_0(p^r)$}.
\newblock {\em Ann. Inst. Fourier (Grenoble)}, 63(3):957--984, 2013.

\bibitem{BuHe1997}
C.~J. Bushnell and G.~Henniart.
\newblock An upper bound on conductors for pairs.
\newblock {\em J. Number Theory}, 65(2):183--196, 1997.

\bibitem{ChSw2024}
I.~Chen and J.~Swidinsky.
\newblock Improved bounds for {S}erre's open image theorem, 2024.

\bibitem{CogdellMichel2004}
J.~Cogdell and P.~Michel.
\newblock On the complex moments of symmetric power {$L$}-functions at {$s=1$}.
\newblock {\em Int. Math. Res. Not.}, (31):1561--1617, 2004.

\bibitem{Co2005}
A.~C. Cojocaru.
\newblock On the surjectivity of the {G}alois representations associated to non-{CM} elliptic curves.
\newblock {\em Canad. Math. Bull.}, 48(1):16--31, 2005.
\newblock With an appendix by Ernst Kani.

\bibitem{DeFrVo2024}
L.~Dembélé, N.~Freitas, and J.~Voight.
\newblock On {G}alois inertial types of elliptic curves over $\mathbb{Q}_\ell$, 2024.

\bibitem{Du1997}
W.~Duke.
\newblock Elliptic curves with no exceptional primes.
\newblock {\em C. R. Acad. Sci. Paris S\'{e}r. I Math.}, 325(8):813--818, 1997.

\bibitem{DuKo2000}
W.~Duke and E.~Kowalski.
\newblock A problem of {L}innik for elliptic curves and mean-value estimates for automorphic representations.
\newblock {\em Invent. Math.}, 139(1):1--39, 2000.
\newblock With an appendix by Dinakar Ramakrishnan.

\bibitem{FoNaTe1992}
E.~Fouvry, M.~Nair, and G.~Tenenbaum.
\newblock L'ensemble exceptionnel dans la conjecture de {S}zpiro.
\newblock {\em Bull. Soc. Math. France}, 120(4):485--506, 1992.

\bibitem{Furio2025}
L.~Furio.
\newblock Effective bounds for adelic {G}alois representations attached to elliptic curves over the rationals, 2025.

\bibitem{FuLo2023}
L.~Furio and D.~Lombardo.
\newblock Serre's uniformity question and proper subgroups of ${C}_{ns}^+(p)$, 2023.

\bibitem{GJ}
S.~Gelbart and H.~Jacquet.
\newblock A relation between automorphic representations of {${\rm GL}(2)$} and {${\rm GL}(3)$}.
\newblock {\em Ann. Sci. \'Ecole Norm. Sup. (4)}, 11(4):471--542, 1978.

\bibitem{Gr2000}
D.~Grant.
\newblock A formula for the number of elliptic curves with exceptional primes.
\newblock {\em Compositio Math.}, 122(2):151--164, 2000.

\bibitem{HIJT}
A.~Hoey, J.~Iskander, S.~Jin, and F.~Trejos~Su\'{a}rez.
\newblock An unconditional explicit bound on the error term in the {S}ato-{T}ate conjecture.
\newblock {\em Q. J. Math.}, 73(4):1189--1225, 2022.

\bibitem{HumphriesThorner2024}
P.~Humphries and J.~Thorner.
\newblock Zeros of {R}ankin-{S}elberg {$L$}-functions in families.
\newblock {\em Compos. Math.}, 160(5):1041--1072, 2024.

\bibitem{Iwaniec1997}
H.~Iwaniec.
\newblock {\em Topics in classical automorphic forms}, volume~17 of {\em Graduate Studies in Mathematics}.
\newblock American Mathematical Society, Providence, RI, 1997.

\bibitem{IwaniecKowalski2004}
H.~Iwaniec and E.~Kowalski.
\newblock {\em Analytic number theory}, volume~53 of {\em American Mathematical Society Colloquium Publications}.
\newblock American Mathematical Society, Providence, RI, 2004.

\bibitem{Jones2013}
N.~Jones.
\newblock Pairs of elliptic curves with maximal {G}alois representations.
\newblock {\em J. Number Theory}, 133(10):3381--3393, 2013.

\bibitem{Ke1982}
M.~A. Kenku.
\newblock On the number of {${\bf Q}$}-isomorphism classes of elliptic curves in each {${\bf Q}$}-isogeny class.
\newblock {\em J. Number Theory}, 15(2):199--202, 1982.

\bibitem{Kim}
H.~H. Kim.
\newblock Functoriality for the exterior square of {${\rm GL}_4$} and the symmetric fourth of {${\rm GL}_2$}.
\newblock {\em J. Amer. Math. Soc.}, 16(1):139--183, 2003.
\newblock With appendix 1 by Dinakar Ramakrishnan and appendix 2 by Henry H. Kim and Peter Sarnak.

\bibitem{Kr1995}
A.~Kraus.
\newblock Une remarque sur les points de torsion des courbes elliptiques.
\newblock {\em C. R. Acad. Sci. Paris S\'{e}r. I Math.}, 321(9):1143--1146, 1995.

\bibitem{Lemos2019}
P.~Lemos.
\newblock Serre's uniformity conjecture for elliptic curves with rational cyclic isogenies.
\newblock {\em Trans. Amer. Math. Soc.}, 371(1):137--146, 2019.

\bibitem{LemosII2019}
P.~Lemos.
\newblock Some cases of {S}erre's uniformity problem.
\newblock {\em Math. Z.}, 292(1-2):739--762, 2019.

\bibitem{Loeffler2017}
D.~Loeffler.
\newblock Images of adelic {G}alois representations for modular forms.
\newblock {\em Glasg. Math. J.}, 59(1):11--25, 2017.

\bibitem{Lo2015}
D.~Lombardo.
\newblock Bounds for {S}erre's open image theorem for elliptic curves over number fields.
\newblock {\em Algebra Number Theory}, 9(10):2347--2395, 2015.

\bibitem{Lombardo2016}
D.~Lombardo.
\newblock An explicit open image theorem for products of elliptic curves.
\newblock {\em J. Number Theory}, 168:386--412, 2016.

\bibitem{Luo1999}
W.~Luo.
\newblock Values of symmetric square {$L$}-functions at {$1$}.
\newblock {\em J. Reine Angew. Math.}, 506:215--235, 1999.

\bibitem{MaWu1993}
D.~W. Masser and G.~W\"{u}stholz.
\newblock Galois properties of division fields of elliptic curves.
\newblock {\em Bull. London Math. Soc.}, 25(3):247--254, 1993.

\bibitem{MaWa2023b}
J.~Mayle and T.~Wang.
\newblock An effective open image theorem for products of principally polarized abelian varieties.
\newblock {\em Journal of Number Theory}, 2025.

\bibitem{Ma1978}
B.~Mazur.
\newblock Rational isogenies of prime degree (with an appendix by {D}. {G}oldfeld).
\newblock {\em Invent. Math.}, 44(2):129--162, 1978.

\bibitem{Mu1999}
M.~R. Murty.
\newblock Bounds for congruence primes.
\newblock In {\em Automorphic forms, automorphic representations, and arithmetic ({F}ort {W}orth, {TX}, 1996)}, volume 66, Part 1 of {\em Proc. Sympos. Pure Math.}, pages 177--192. Amer. Math. Soc., Providence, RI, 1999.

\bibitem{NeTh2021}
J.~Newton and J.~A. Thorne.
\newblock Symmetric power functoriality for holomorphic modular forms, {II}.
\newblock {\em Publ. Math. Inst. Hautes \'{E}tudes Sci.}, 134:117--152, 2021.

\bibitem{Ramakrishnan}
D.~Ramakrishnan.
\newblock Modularity of the {R}ankin-{S}elberg {$L$}-series, and multiplicity one for {${\rm SL}(2)$}.
\newblock {\em Ann. of Math. (2)}, 152(1):45--111, 2000.

\bibitem{RamakrishnanYang}
D.~Ramakrishnan and L.~Yang.
\newblock A constraint for twist equivalence of cusp forms on {${\rm GL}(n)$}.
\newblock {\em Funct. Approx. Comment. Math.}, 65(1):105--117, 2021.

\bibitem{Ri1975}
K.~A. Ribet.
\newblock On {$l$}-adic representations attached to modular forms.
\newblock {\em Invent. Math.}, 28:245--275, 1975.

\bibitem{Ro1994}
D.~E. Rohrlich.
\newblock Elliptic curves and the {W}eil-{D}eligne group.
\newblock In {\em Elliptic curves and related topics}, volume~4 of {\em CRM Proc. Lecture Notes}, pages 125--157. Amer. Math. Soc., Providence, RI, 1994.

\bibitem{Se1972}
J.-P. Serre.
\newblock Propri\'{e}t\'{e}s galoisiennes des points d'ordre fini des courbes elliptiques.
\newblock {\em Invent. Math.}, 15(4):259--331, 1972.

\bibitem{Se1981}
J.-P. Serre.
\newblock Quelques applications du th\'{e}or\`eme de densit\'{e} de {C}hebotarev.
\newblock {\em Inst. Hautes \'{E}tudes Sci. Publ. Math.}, (54):323--401, 1981.

\bibitem{SeTa1968}
J.-P. Serre and J.~Tate.
\newblock Good reduction of abelian varieties.
\newblock {\em Ann. of Math. (2)}, 88:492--517, 1968.

\bibitem{ShSHWa2021}
A.~N. Shankar, A.~Shankar, and X.~Wang.
\newblock Large families of elliptic curves ordered by conductor.
\newblock {\em Compos. Math.}, 157(7):1538--1583, 2021.

\bibitem{Si2009}
J.~H. Silverman.
\newblock {\em The arithmetic of elliptic curves}, volume 106 of {\em Graduate Texts in Mathematics}.
\newblock Springer, Dordrecht, second edition, 2009.

\bibitem{Soundararajan2010}
K.~Soundararajan.
\newblock Weak subconvexity for central values of {$L$}-functions.
\newblock {\em Ann. of Math. (2)}, 172(2):1469--1498, 2010.

\bibitem{JT_Siegel}
J.~{Thorner}.
\newblock {Exceptional zeros of Rankin-Selberg $L$-functions and joint Sato-Tate distributions}.
\newblock {\em arXiv e-prints}, page arXiv:2404.06482, Apr. 2024.

\bibitem{Wa2008}
M.~Watkins.
\newblock Some heuristics about elliptic curves.
\newblock {\em Experiment. Math.}, 17(1):105--125, 2008.

\bibitem{Xi2024}
S.~Y. Xiao.
\newblock Elliptic curves with a rational 2-torsion point ordered by conductor and the boundedness of average rank.
\newblock {\em Adv. Math.}, 446:Paper No. 109681, 40, 2024.

\bibitem{Zhao2023}
S.~Zhao.
\newblock On {S}iegel {Z}eros of {S}ymmetric {P}ower {$L$}-functions.
\newblock {\em arXiv preprint arXiv:2301.07869}, 2023.

\end{thebibliography}

\end{document}